\documentclass[12pt]{article}

\usepackage[margin=1in]{geometry}
\usepackage{amsmath,amssymb,amsthm}
\usepackage{graphicx}
\usepackage{xcolor}
\usepackage{xurl}
\usepackage[colorlinks=true,linkcolor=blue,urlcolor=blue,citecolor=blue]{hyperref}
\graphicspath{{figures/}}
\hypersetup{pdftitle={The M/M(D)/c queue with dependent servers: A canonical Markovian extension of M/M/c},pdfauthor={Qihui Bu, Suman Thapa, Yiqiang Q. Zhao}}

\numberwithin{equation}{section}
\newtheorem{theorem}{Theorem}[section]
\newtheorem{lemma}{Lemma}[section]
\newtheorem{corollary}{Corollary}[section]
\newtheorem{proposition}{Proposition}[section]
\theoremstyle{definition}
\newtheorem{definition}{Definition}[section]
\newtheorem{example}{Example}[section]

\theoremstyle{remark}
\newtheorem{remark}{Remark}[section]

\newif\ifshowchanges
\showchangesfalse
\newcommand{\chg}[1]{\ifshowchanges{\color{blue}#1}\else#1\fi}
\newenvironment{chgblock}{\begingroup\ifshowchanges\color{blue}\fi}{\endgroup}

\makeatletter
 \newif\if@borderstar
 \def\bordermatrix{\@ifnextchar*{%
 \@borderstartrue\@bordermatrix@i}{\@borderstarfalse\@bordermatrix@i*}%
 }
 \def\@bordermatrix@i*{\@ifnextchar[{\@bordermatrix@ii}{\@bordermatrix@ii[()]}}
 \def\@bordermatrix@ii[#1]#2{%
 \begingroup
 \m@th\@tempdima8.75\p@\setbox\z@\vbox{%
 \def\cr{\crcr\noalign{\kern 2\p@\global\let\cr\endline }}%
 \ialign {$##$\hfil\kern 2\p@\kern\@tempdima & \thinspace %
 \hfil $##$\hfil && \quad\hfil $##$\hfil\crcr\omit\strut %
 \hfil\crcr\noalign{\kern -\baselineskip}#2\crcr\omit %
 \strut\cr}}%
 \setbox\tw@\vbox{\unvcopy\z@\global\setbox\@ne\lastbox}%
 \setbox\tw@\hbox{\unhbox\@ne\unskip\global\setbox\@ne\lastbox}%
 \setbox\tw@\hbox{%
 $\kern\wd\@ne\kern -\@tempdima\left\@firstoftwo#1%
 \if@borderstar\kern2pt\else\kern -\wd\@ne\fi%
 \global\setbox\@ne\vbox{\box\@ne\if@borderstar\else\kern 2\p@\fi}%
 \vcenter{\if@borderstar\else\kern -\ht\@ne\fi%
 \unvbox\z@\kern-\if@borderstar2\fi\baselineskip}%
 \if@borderstar\kern-2\@tempdima\kern2\p@\else\,\fi\right\@secondoftwo#1 $%
 }\null \;\vbox{\kern\ht\@ne\box\tw@}%
 \endgroup
 }
\makeatother

\title{The $M/M_D/c$ queue with dependent servers:\\ A canonical Markovian extension of $M/M/c$}
\author{%
Qihui Bu\thanks{School of Communications and Information Engineering, Nanjing University of Posts and Telecommunications. Email: \href{mailto:buqihui@njupt.edu.cn}{buqihui@njupt.edu.cn}.}
\and
Suman Thapa\thanks{School of Mathematics and Statistics, Carleton University. Email: \href{mailto:sumanthapa@cmail.carleton.ca}{sumanthapa@cmail.carleton.ca}.}
\and
Yiqiang Q. Zhao\thanks{School of Mathematics and Statistics, Carleton University. Email: \href{mailto:zhao@math.carleton.ca}{zhao@math.carleton.ca}.}}
\date{}

\begin{document}
\maketitle

\begin{abstract}
The classical $M/M/c$ queue is the canonical Markovian model for a
single-node queue with Poisson arrivals and homogeneous independent servers.
We identify its natural counterpart when the servers are dependent.
For homogeneous Markovian labelled service-completion mechanisms
(retaining the identities of the completing servers), consistency under
restriction to subsets of busy servers is equivalent to an independent
Poisson shock representation, yielding homogeneous Marshall--Olkin
multivariate exponential residual-service laws. This characterization
identifies the $M/M_D/c$ queue as the canonical Markovian extension of
$M/M/c$ to dependent homogeneous servers, with a Markovian queue-length
process. For heterogeneous servers, augmenting boundary states with
busy-server identities yields a Markovian queue-state process, as in the
classical independent-server setting. For arbitrary $c$ and fixed idle-set
assignment, we derive an explicit matrix-analytic stationary solution and
departure and customer-delay laws. Waiting time conditional on delay is
exponential, and the stationary sojourn distribution has a phase-type
representation with at most $c+1$ phases. A three-server comparison with
identical bivariate residual-service laws demonstrates how higher-order
dependence can change waiting-time behavior.
\end{abstract}



\medskip
\section{Introduction}

The classical $M/M/c$ queue is the canonical Markovian model for a
single-node queue with Poisson arrivals and $c$ homogeneous independent
servers. The memoryless property of exponential service times makes
the queue-length process a continuous-time Markov chain, without
the need to record service ages or residual service times.
When the servers are dependent, however, service completions may occur
simultaneously, and the dependence mechanism must remain well defined
as customers enter and leave service. This raises a fundamental question:
what is the natural Markovian counterpart of the $M/M/c$ queue when its
servers are dependent?

We answer this question within a natural class of service-completion
mechanisms. Assume that labelled service-completion rates are Markovian,
where ``labelled'' means that the identities of the completing servers
are retained, and that the rates depend only on the set of busy servers.
We impose \emph{consistency under restriction}: the marginal completion
mechanism of a subset of servers is unaffected by which additional
servers are busy.
We establish that this consistency condition is equivalent to a
representation by mutually independent Poisson shocks acting on server
subsets. For homogeneous servers, the shock rates depend only on
subset size, so the residual completion times have a homogeneous
Marshall--Olkin multivariate exponential distribution
\cite{Marshall-Olkin:1967a}. Under these conditions, the $M/M_D/c$
queue is therefore the canonical Markovian extension of the classical
$M/M/c$ queue to dependent homogeneous servers. Its queue-length
process remains Markovian; the independent-server $M/M/c$ queue is
recovered when all shared-shock rates vanish.
Importantly, the characterization concerns \emph{labelled}
completion mechanisms: Markovianity of queue length alone does not
imply a Marshall--Olkin representation.

For heterogeneous servers, queue length by itself generally does
not determine transition rates when some servers are idle, because
the rates depend on which servers are busy. As in the classical
independent heterogeneous-server model, we augment the boundary
states by the identities of busy servers. This gives a Markovian
queue-state process while retaining customer count as the state
description whenever all servers are busy.

We analyze this framework for arbitrary $c$, Poisson arrivals,
first-come-first-served (FCFS) waiting, and sequential assignment
governed by a fixed rule for each idle-server set.
Independent Poisson shocks act persistently on server labels;
a shock completes the affected services before waiting customers
refill the available servers. The corresponding Marshall--Olkin
laws concern residual service-completion times measured at a common
observation time. Since customers may begin service at different
times, they need not describe the joint distribution of complete
service durations measured from different initiation epochs.

This construction differs from dependence along a tandem queue
\cite{Mitchell et al:1977} and dependence among members of an arrival
batch in an infinite-server queue~\cite{Pang-Whitt:2012}.
Thapa and Zhao~\cite{Thapa-Zhao:2026} introduced the $M/M_D/2$ model
and obtained its stationary distribution; their review gives further
motivation and related models. The present paper extends the analysis
to arbitrary $c$, heterogeneous servers, and fixed idle-set assignment,
and connects the queue-state analysis to departure and customer-delay
laws. A preliminary version appeared on arXiv~\cite{Bu-Thapa-Zhao:2026}.
This manuscript revises and expands that same work, replacing its
scalar-tail consistency approach by a matrix-analytic analysis.

Our probabilistic foundation combines a sample-path construction
and stability criterion with the characterization of the service
mechanism through consistency of labelled completion rates.
The latter specifies the requirement beyond queue-state Markovianity
and complements the characterization through Markovian survival
subvectors in Brigo et al.\ \cite{Brigo-Mai-Scherer:2016}.
The classical distributional characterization of the Marshall--Olkin
law is recalled separately.

Our stationary analysis uses a quasi-birth-and-death (QBD) formulation.
\begin{chgblock}
Matrix-analytic and matrix-geometric methods for structured Markov chains
are standard; see, for example, Neuts~\cite{Neuts:1981} and He~\cite{He:2014}.
Rank-one QBDs with explicit matrix-geometric solutions are established
in particular classes~\cite{Ramaswami-Latouche:1986,vanLeeuwaarden-Winands:2006}.
\end{chgblock}
We evaluate this structure for the dependent-server queue and use it
throughout the performance analysis. Departures are treated through
established marked-generator methods~\cite{Lucantoni:1991}; delays
use the FCFS allocation of simultaneous service openings.
Classical bulk-service waiting-time analysis~\cite{Medhi:1975}
provides related context, although here the openings arise from
shocks to occupied servers followed by immediate refill.
The resulting contributions are as follows.

\begin{enumerate}
\begin{chgblock}
\item \emph{Probabilistic construction and characterization.}
We give a sample-path construction and stability criterion. Within the
class of Markovian labelled service-completion mechanisms, consistency
under restriction is equivalent to a representation by mutually independent
Poisson shock processes, whose rates are uniquely recovered from fully busy
completions. In the homogeneous case this yields the canonical
Marshall--Olkin dependent-server analogue of $M/M/c$, while identifying
the additional structure required beyond queue-length Markovianity.

\item \emph{An explicit matrix-analytic stationary solution.}
The minimal QBD rate matrix has only one nonzero row,
$(r,r^2,\ldots,r^c)$, where a strictly monotone scalar equation
determines $r\in(0,1)$. This gives a scalar geometric tail and a
finite boundary generator that satisfy the boundary and tail equations
together. Homogeneous boundary probabilities admit a scalar backward
recursion; the heterogeneous system uses the $2^c-1$ proper busy-server
subsets and includes zero shock rates and inaccessible states.

\item \emph{Exact output and customer-delay laws.}
Explicit departure marks give batch and departure-epoch measures, a finite
phase-type interdeparture law, a marked Markov-renewal description, and
finite-time count variability. Conservation yields the Poisson input's
long-run count-variance rate and central limit theorem without implying
Poisson output; even for independent heterogeneous servers, the output can
depend on the idle-server assignment rule. For customers, waiting time
conditional on delay is exactly exponential, and ordered refill probabilities
give the joint waiting and service law and a stationary sojourn
representation with at most $c+1$ phases.

\item \emph{A comparison isolating higher-order dependence.}
For two three-server families with identical bivariate residual-service
laws, we prove a strict ordering of waiting times conditional on delay.
Together with the boundary delay probabilities, the decay rates give
an exact criterion for crossing of the unconditional waiting tails.
Thus matching all bivariate residual laws does not determine waiting
performance. Numerical comparisons also illustrate heterogeneity,
assignment, and output variability.
\end{chgblock}
\end{enumerate}

Section~\ref{sec:MO-distribution} recalls the Marshall--Olkin distribution.
Section~\ref{sec:queue-model} constructs the model and derives its
stationary distribution. Section~\ref{sec:Markovian-characterization}
characterizes the service mechanism through completion-rate consistency.
Sections~\ref{sec:departures} and~\ref{sec:waiting-sojourn} analyze
departures and customer delays. Section~\ref{sec:numerical-study}
presents the dependence comparison and numerical validation,
followed by concluding remarks.

\section{The Marshall--Olkin multivariate exponential distribution}
\label{sec:MO-distribution}

We recall the Marshall--Olkin distribution (MO-MVED), the marginal
rates used in the queue generator, and its lack-of-memory
characterization. The latter is distributional; the converse for
queue-completion mechanisms in Section~\ref{sec:Markovian-characterization}
also requires completion-rate consistency.

\begin{definition}\label{def:lack-of-memory}
Random vector $(X_1,\ldots,X_c)$ is said to satisfy the (multivariate) lack-of-memory property in weak sense if, for every nonempty subset
$A=\{i_1,\ldots,i_\ell\}\subseteq\{1,\ldots,c\}$,
\[
P(X_{i_j}>x_{i_j}+t,\ j=1,\ldots,\ell
\mid
X_{i_j}>x_{i_j},\ j=1,\ldots,\ell)
=
P(X_{i_j}>t,\ j=1,\ldots,\ell)
\]
for all $t\ge0$ and $x_{i_j}\ge0$, $j=1,\ldots,\ell$.
\end{definition}

\begin{definition}[\cite{Marshall-Olkin:1967a}] \label{def-MO-MVE}
\begin{chgblock}
For a positive integer $c$, a random vector $\mathbf{X}=(X_1,\ldots,X_c)\in\mathbb R_+^c$
is said to follow an MO-MVED if its joint survival function $\bar H$ is given by
\end{chgblock}
\begin{equation}\label{eqn:MO-MVE}
\bar{H}(\mathbf{x}) = \mathbb{P}(X_1 > x_1, \ldots, X_c > x_c) = \exp\left( - \sum_{\chg{A \in \mathcal{P}}} \mu_A \max_{i \in A} x_i \right),
\end{equation}
where
$\mathcal{P}=2^{\{1,\ldots,c\}}\setminus\{\emptyset\}$ is the set
of all nonempty subsets of $\{1,\ldots,c\}$. The finite parameters
$\mu_A\ge0$ are the rates of shocks affecting the components in $A$.
\begin{chgblock}
We require
\[
 \sum_{\chg{A\in\mathcal P:\,i\in A}}\mu_A>0,
 \qquad i=1,\ldots,c,
\]
so each component has a proper exponential distribution.
\end{chgblock}
\end{definition}

\begin{lemma}[Marginal rates and lack of memory]
\label{lem:closure}\label{rem:part-of-lack-of-memory-2}
For a vector with survival function \eqref{eqn:MO-MVE} and any
nonempty $T\subseteq\{1,\ldots,c\}$, the marginal vector
$X_T=(X_i)_{i\in T}$ is Marshall--Olkin with survival function
\begin{equation}\label{eqn:barHofT-2}
 \bar H_T(x_T)
 =\exp\left\{-\sum_{\emptyset\ne B\subseteq T}
                  \mu_B^T\max_{i\in B}x_i\right\},
\end{equation}
where
\begin{equation}\label{eqn:marginal-rate}
 \mu_B^T=\sum_{\chg{A\in\mathcal P:\,A\cap T=B}}\mu_A,
 \qquad\emptyset\ne B\subseteq T.
\end{equation}
Each such marginal satisfies the common-shift identity
\begin{equation}\label{eqn:H-Bar}
 \bar H_T(x_T+t\mathbf1_T)
 =\bar H_T(x_T)\bar H_T(t\mathbf1_T),\qquad x_T\ge0,\ t\ge0,
\end{equation}
and hence the weak lack-of-memory property in
Definition~\ref{def:lack-of-memory}.
\end{lemma}

\begin{proof}
Set $x_j=0$ for $j\notin T$ in \eqref{eqn:MO-MVE}, using
$\mathbb P(X_j>0)=1$. Shocks disjoint from $T$ contribute zero;
grouping the others by $B=A\cap T$ gives
\eqref{eqn:barHofT-2}--\eqref{eqn:marginal-rate}.
The identity $\max_{i\in B}(x_i+t)=\max_{i\in B}x_i+t$ then
gives \eqref{eqn:H-Bar}. Dividing by the positive marginal survival
function yields the stated conditional-probability identity.
\end{proof}

\begin{example}[A three-variable marginal]
Write $\mu_i=\mu_{\{i\}}$, $\mu_{ij}=\mu_{\{i,j\}}$, and
$\mu_{123}=\mu_{\{1,2,3\}}$. For $c=3$ and $T=\{1,3\}$,
\[
 (\mu_1^T,\mu_3^T,\mu_{13}^T)
 =(\mu_1+\mu_{12},\ \mu_3+\mu_{23},\ \mu_{13}+\mu_{123}),
\]
so
\[
 \bar H_T(x_1,x_3)
 =\exp\{-(\mu_1+\mu_{12})x_1-(\mu_3+\mu_{23})x_3
          -(\mu_{13}+\mu_{123})\max(x_1,x_3)\}.
\]
Shocks involving the omitted component still contribute through
their intersection with $T$; retaining only the original rates
indexed by subsets of $T$ would generally give incorrect marginals.
\end{example}

The converse characterizes the MO-MVED; see
\cite{Marshall-Olkin:1967a,BernhartEtAl2015}. A self-contained proof,
including recovery and nonnegativity of the shock rates, is given in
Appendix~\ref{app:MO-characterization}.

\begin{theorem}[Marshall--Olkin characterization] \label{thm:MO-characterization}
Let $\mathbf X$ take values in $[0,\infty)^c$. Assume
that $\bar F_B$ is positive and continuous on $\mathbb R_+^{|B|}$ for
every nonempty $B\subseteq\{1,\ldots,c\}$ and that
$\mathbb P(X_i>0)=1$ for all $i$. Then $\mathbf X$ satisfies the weak
multivariate lack-of-memory property for every marginal subvector if
and only if its survival function has the Marshall--Olkin form
\[
 \bar F(x_1,\ldots,x_c)
 =\exp\left\{-\sum_{\emptyset\ne A\subseteq\{1,\ldots,c\}}
                   \mu_A\max_{i\in A}x_i\right\},
\]
where the rates $\mu_A\ge0$ are finite and uniquely determined, and
\chg{$\sum_{\chg{A\in\mathcal P:\,i\in A}}\mu_A>0$ for every $i$.}
\end{theorem}

\section{Model description and Markov chain}
\label{sec:queue-model}

Consider $c\ge2$ servers, an infinite waiting room, and Poisson arrivals
of rate $\lambda>0$, independent of the Marshall--Olkin service shocks.
Service is nonpreemptive and work-conserving \chg{(no server is idle while a customer is waiting)}, with first-come-first-served
(FCFS) waiting and the sequential idle-server assignment rule specified below.

Write $\mathcal I=\{1,\ldots,c\}$. The queue-state process $Y(t)$ has
state space
\begin{equation}\label{eqn:state-space-MMDc}
S=\{0\}\cup\{s\subseteq\mathcal I:1\le |s|\le c-1\}
  \cup\{c,c+1,\ldots\}.
\end{equation}
State $0$ is empty; a proper subset $s$ records the busy servers with
no waiting customer; an integer $n\ge c$ records the total customer
count with all servers busy. Denote the total count by $x(y)$ and the
busy set by $B(y)$:
\[
 (x(0),B(0))=(0,\emptyset),\quad
 (x(s),B(s))=(|s|,s),\quad
 (x(n),B(n))=(n,\mathcal I)\quad(n\ge c).
\]
Thus $X(t)=x(Y(t))$. \chg{When the underlying state is clear from the context,
we write simply $B$ for $B(y)$.} Example~\ref{exa:stability-3} illustrates the
states, shock rates, and assignments for $c=3$.

\subsection{Marshall--Olkin shock system and sample-path construction}
\label{subsec:shock-construction}

\paragraph{Driving processes and random assignment.}
Let $\mathcal P=2^{\mathcal I}\setminus\{\emptyset\}$.
The driving processes are mutually independent Poisson processes:
$N_A$ of rate $\mu_A\ge0$ for each $A\in\mathcal P$, and the arrival
process $N_\lambda$ of rate $\lambda$. All rates $\mu_A$ and $\lambda$
are finite, and each server has positive marginal completion rate
\begin{equation}\label{eq:server-completion-rate}
\nu_i:=\sum_{\chg{A\in\mathcal P:\,i\in A}}\mu_A
      =\mu_i^{\{i\}}>0,\qquad i\in\mathcal I.
\end{equation}
For every nonempty idle-server set $E$, fix probabilities
\begin{equation}\label{eq:assignment-probabilities}
p_i(E)\ge0,\qquad i\in E,\qquad \sum_{i\in E}p_i(E)=1.
\end{equation}
A customer seeing idle set $E$ selects server $i\in E$ with probability
$p_i(E)$. Multiple assignments occur instantaneously in FCFS order,
removing each selected server from $E$ before the next selection.
The same probabilities apply at arrivals and after completions;
zero and unit probabilities include deterministic rules.

At each arrival or shock event, use a fresh mark $(U_1,\ldots,U_c)$
of independent uniform $(0,1)$ variables, independent across events
and of all driving processes. For the $j$th assignment, order the
current idle set as $E=\{i_1<\cdots<i_q\}$ and select $i_r$ when
\[
 \sum_{\ell=1}^{r-1}p_{i_\ell}(E)<U_j
 \le\sum_{\ell=1}^{r}p_{i_\ell}(E).
\]
Then update $E$ and use $U_{j+1}$ for the next customer. At most $c$
assignments occur at one event, and unused marks are ignored.
The initial state is independent of the marked driving processes.
Let $\mathcal F_t$ be the usual completed filtration generated by
$Y(0)$ and the driving events and marks revealed by time $t$.

\paragraph{Event rules.}
At each driving event:
\begin{itemize}
\item An arrival selects an idle server by
\eqref{eq:assignment-probabilities}, or joins the waiting queue if all
servers are busy.
\item A type-$A$ shock completes the services at all servers in
$A\cap B(Y(t^-))$. If this intersection is empty, no customer departs
and the queue state is unchanged.
\item After any single or simultaneous completion, waiting customers
immediately refill available servers in FCFS order until the queue
or idle set is empty. Newly assigned customers are affected only by
subsequent shocks.
\end{itemize}
The shock processes run while servers are busy or idle and are never
restarted at service initiation. Distinct driving processes almost
surely have no common event times; simultaneous departures therefore
come from one shock.

\paragraph{Assignment probabilities after a shock.}
For $E\ne\emptyset$, $0\le k\le|E|$, and $F\subseteq E$ with $|F|=k$,
the probability that $k$ sequential assignments select exactly $F$ is
\begin{equation}\label{eq:assignment-subset-probability}
K_{E,k}(F)=
\sum_{(i_1,\ldots,i_k)\in\operatorname{Perm}(F)}
\prod_{j=1}^{k}
p_{i_j}\bigl(E\setminus\{i_1,\ldots,i_{j-1}\}\bigr),
\end{equation}
where $\operatorname{Perm}(F)$ denotes all orderings of $F$.
These probabilities sum to one; in particular,
$K_{E,0}(\emptyset)=K_{E,|E|}(E)=1$.

Identify $\emptyset$ with state $0$ and $\mathcal I$ with state $c$.
From a proper subset $s$, an arrival produces $s\cup\{i\}$ with
probability $p_i(\mathcal I\setminus s)$, and a type-$A$ shock produces
$s\setminus A$ without refilling. From $n\ge c$, an arrival produces
$n+1$, whereas a type-$A$ shock gives
\begin{equation}\label{eq:shock-update-map}
Y(t)=
\begin{cases}
n-|A|,&n-|A|\ge c,\\
(\mathcal I\setminus A)\cup F,&n-|A|<c,
\end{cases}
\end{equation}
where, in the second case, the $n-c<|A|$ waiting customers select
$F\subseteq A$ with probability $K_{A,n-c}(F)$. Servers outside $A$
retain their customers. In either case,
\begin{equation}\label{eq:count-updates}
X(t)-X(t^-)=
\begin{cases}
1,&\text{at an arrival},\\
-|A\cap B(Y(t^-))|,&\text{at a type-$A$ shock}.
\end{cases}
\end{equation}

\begin{remark}[Special assignments]
\label{rem:assignment-special-cases}
Uniform assignment has $p_i(E)=1/|E|$; the smallest-index rule has
$p_{\min E}(E)=1$. For $c=2$, setting $p_1(\{1,2\})=p$ recovers
\cite{Thapa-Zhao:2026}: the same $p$ governs an empty-system arrival
and refilling after a simultaneous completion from state $3$.
The choice $p=1$ is used below for the independent-service comparison
with Singh~\cite{Singh:1970}.
\end{remark}

\paragraph{Residual completion times.}
For any almost surely finite $\mathcal F_t$-stopping time $\tau$, define
the time from $\tau$ to the next shock involving server $i$ by
\[
T_i(\tau)=\inf\left\{u>0:
\sum_{A\in\mathcal P:\,i\in A}
  [N_A(\tau+u)-N_A(\tau)]\ge1\right\}.
\]
\begin{chgblock}
By the strong Markov property of the vector of mutually independent Poisson
driving processes $(N_\lambda,(N_A)_{A\in\mathcal P})$, together with the
independence of the uniform $(0,1)$ assignment marks, conditional on
$\mathcal F_\tau$ the post-$\tau$ shock processes remain independent Poisson
processes with respective rates $\mu_A$. Hence, for every nonempty
$B\subseteq\mathcal I$ and $t_i\ge0$,
\end{chgblock}
\begin{equation}\label{eq:conditional-residual-MO}
\mathbb P\bigl(T_i(\tau)>t_i,\ i\in B\mid\mathcal F_\tau\bigr)
=\exp\left\{-\sum_{\chg{A\in\mathcal P:\,A\cap B\ne\emptyset}}
  \mu_A\max_{i\in A\cap B}t_i\right\}.
\end{equation}
The event requires no type-$A$ shock for
$\max_{i\in A\cap B}t_i$ time units after $\tau$; independence gives
the marginal survival law of Lemma~\ref{lem:closure}.
For busy servers, $T_i(\tau)$ are the customers' residual completion
times. At initiation, a service at $i$ is exponential with rate $\nu_i$.
The joint law concerns residual times at a common observation time,
not full service durations begun at different times.

\paragraph{Markov property and transition rates.}
\begin{chgblock}
For a busy-server set $s$, let $\mu_D^s$ denote the rate at which a shock
causes simultaneous service completions at precisely the busy servers in
$D$, where $\emptyset\ne D\subseteq s$, and let $\alpha(s)$ denote the
total rate of shocks that cause at least one service completion when the
busy set is $s$. Thus,
\[
\mu_D^s=\sum_{A\in\mathcal P:\,A\cap s=D}\mu_A,
\qquad \emptyset\ne D\subseteq s,
\qquad
\alpha(s)=\sum_{A\in\mathcal P:\,A\cap s\ne\emptyset}\mu_A,
\]
with $\alpha(\emptyset)=0$, as in \eqref{eqn:marginal-rate}.
Define the aggregate rate of shocks involving exactly $\ell$ servers by
\begin{equation}\label{eq:shock-size-rates}
 R_\ell:=\sum_{A\in\mathcal P:\,|A|=\ell}\mu_A,
 \qquad 1\le\ell\le c.
\end{equation}
\end{chgblock}

\begin{theorem}\label{thm:MMDc-CTMC}
The construction defines a nonexplosive, time-homogeneous continuous-time
Markov chain $Y(t)$. Its generator $Q=(q_{y,z})$ has the following
off-diagonal entries. For a proper subset state $s$, including $0$,
\begin{align}
q_{s,s\cup\{i\}}&=\lambda p_i(\mathcal I\setminus s),
&&i\in\mathcal I\setminus s,\label{eq:generator-subset-arrival}\\
q_{s,s\setminus D}&=\mu_D^s,
&&\emptyset\ne D\subseteq s.\label{eq:generator-subset-departure}
\end{align}
For an integer state $n\ge c$,
\begin{align}
q_{n,n+1}&=\lambda,\label{eq:generator-full-arrival}\\
q_{n,n-\ell}&=R_\ell,
\qquad 1\le\ell\le c,\quad n-\ell\ge c.
\label{eq:generator-full-departure}
\end{align}
For transitions from $n\ge c$ to a proper subset state $s$,
\begin{equation}\label{eq:generator-refill-boundary}
q_{n,s}=
\sum_{\substack{A\in\mathcal P:\ \mathcal I\setminus A\subseteq s\\
                         |A|=n-|s|}}
\mu_A K_{A,n-c}(s\cap A).
\end{equation}
Empty sums are zero. All other off-diagonal entries are zero, and
\begin{equation}\label{eq:generator-diagonal}
q_{y,y}=-\bigl[\lambda+\alpha(B(y))\bigr].
\end{equation}
Equivalently, for bounded functions $f$,
\begin{equation}\label{eq:MMDc-generator}
(Qf)(y)=\sum_{z\ne y}q_{y,z}[f(z)-f(y)].
\end{equation}
\end{theorem}

\begin{proof}
The superposition of all driving processes has finite rate
$\kappa=\lambda+\sum_{A\in\mathcal P}\mu_A$. There are therefore
only finitely many driving events, and hence queue-state changes, in any
bounded time interval. The event rules and the marks define a unique
right-continuous sample path; at most $c$ instantaneous assignments are
needed at any event. Given $Y(t)$, the future queue states are determined
by the future driving events and their marks. These are independent of
$\mathcal F_t$, and their laws do not depend on $t$. The assignment
probabilities use only the current idle-server set. No additional past
information is therefore needed to determine the conditional law of
future states, which proves the time-homogeneous Markov property.

Conditioning on the next driving event gives
\eqref{eq:generator-subset-arrival}--\eqref{eq:generator-full-departure}.
For \eqref{eq:generator-refill-boundary}, a type-$A$ shock from state
$n\ge c$ leaves all servers in $\mathcal I\setminus A$ busy and
leaves $n-|A|$ customers in the system. To reach a proper subset $s$,
we must therefore have $\mathcal I\setminus A\subseteq s$ and
$|A|=n-|s|$. The $n-c$ waiting customers must select precisely
$s\cap A$, with probability $K_{A,n-c}(s\cap A)$. Summing the
corresponding rates gives the stated formula. An arrival always changes
the state, whereas a shock does so exactly when it involves a busy
server. This yields \eqref{eq:generator-diagonal} and completes the
generator specification.
\end{proof}

The exit rate is bounded by $\kappa$, and assignment affects only
\eqref{eq:generator-subset-arrival} and
\eqref{eq:generator-refill-boundary}, not the fully busy repeating rates.

\paragraph{Departure counts and stability.}
Let $D(t)$ count the customers departing during $(0,t]$, with simultaneous
departures counted separately. Then
\begin{equation}\label{eq:departure-count-construction}
D(t)=\sum_{A\in\mathcal P}\int_{(0,t]}
 |A\cap B(Y(u^-))|\,dN_A(u),
\qquad X(t)=X(0)+N_\lambda(t)-D(t).
\end{equation}
The instantaneous mean customer-departure rate in state $y$ is
\begin{equation}\label{eq:customer-departure-rate}
\delta(y)=\sum_{A\in\mathcal P}|A\cap B(y)|\mu_A
         =\sum_{i\in B(y)}\nu_i.
\end{equation}
Thus the full-busy service capacity is
\begin{equation}\label{eq:full-busy-service-capacity}
\Theta=\sum_{i=1}^c\nu_i
      =\sum_{A\in\mathcal P}|A|\mu_A.
\end{equation}
The factor $|A|$ is essential: service capacity counts customers, not
shock events.

Some subset states may be inaccessible when shock rates or assignment
probabilities vanish. We therefore consider the communicating class
$S_0$ reachable from state
$0$. It is closed and irreducible: every reachable state can return to
$0$ along a finite sequence of positive-rate shocks without intervening
arrivals, since each $\nu_i>0$. The same argument shows that every state
in $S$ can reach $0$, so no other closed communicating class exists.
If all singleton-shock rates are positive, $S_0=S$. Stability below means
positive recurrence of $Y(t)$ on $S_0$; any states outside $S_0$ have zero
stationary probability.

\begin{lemma}[Stability]\label{lem:stability}
The $M/M_D/c$ queue is stable if and only if
\begin{equation}\label{eq:MMDc-stability}
\lambda<\Theta
       =\sum_{i=1}^c\mu_i^{\{i\}}
       =\sum_{A\in\mathcal P}|A|\mu_A.
\end{equation}
Under this condition it has a unique stationary probability distribution,
strictly positive on $S_0$.
\end{lemma}

\begin{proof}
Let $V(y)=x(y)$. The bounded jump sizes and finite driving rate allow
\eqref{eq:MMDc-generator} to be applied to $V$, giving
\[
(QV)(y)=\lambda-\delta(y).
\]
If $x(y)\ge c$, all servers are busy and
$(QV)(y)=\lambda-\Theta$. For $\lambda<\Theta$, this is a strictly
negative constant outside the finite set $\{y:x(y)<c\}$. Uniformization
at rate $\kappa$ gives an irreducible discrete-time chain on $S_0$ with
transition matrix $P=I+Q/\kappa$ and drift
$PV-V=(\lambda-\Theta)/\kappa<0$ outside that set. On the finite set,
$PV$ is finite. The Foster drift criterion therefore gives positive
recurrence and uniqueness of the stationary distribution.

Conversely, suppose that a stationary distribution $\pi$ exists. Since
$0$ belongs to the unique closed communicating class, $\pi_0>0$.
Use the bounded functions $V_K(y)=\min\{x(y),K\}$, for which
$\sum_y\pi_y(QV_K)(y)=0$. Equation~\eqref{eq:count-updates} gives
$|(QV_K)(y)|\le\lambda+\Theta$ uniformly in $K$ and $y$.
As $K\to\infty$, dominated convergence therefore yields
\[
0=\sum_y\pi_y(QV)(y)
 =\lambda-\sum_y\pi_y\delta(y).
\]
Because $\delta(0)=0$ and $\delta(y)\le\Theta$,
\[
\lambda=\sum_y\pi_y\delta(y)
       \le(1-\pi_0)\Theta<\Theta.
\]
This proves necessity, including exclusion of the critical case
$\lambda=\Theta$, without assuming a finite stationary mean customer
count or using a geometric-tail representation.
\end{proof}

\begin{example}[Three-server states, shocks, and refilling]
\label{exa:stability-3}\label{exa:states-3}
\normalfont
For $c=3$,
\[
 S=\{0,(1),(2),(3),(1,2),(1,3),(2,3),3,4,5,\ldots\},
\]
where $(i)$ and $(i,j)$ denote busy subsets $\{i\}$ and $\{i,j\}$,
with $i\in\{1,2,3\}$ and $1\le i<j\le3$; integer states $n\ge3$
have all servers busy and $n-3$ customers waiting.
Write $\mu_i=\mu_{\{i\}}$, $\mu_{ij}=\mu_{\{i,j\}}$, and
$\mu_{123}=\mu_{\{1,2,3\}}$. For $\{i,j,k\}=\mathcal I$,
\[
 \nu_i=\mu_i+\mu_{ij}+\mu_{ik}+\mu_{123},\qquad
 R_1=\mu_1+\mu_2+\mu_3,\quad
 R_2=\mu_{12}+\mu_{13}+\mu_{23},\quad R_3=\mu_{123}.
\]
Stability is $\lambda<R_1+2R_2+3R_3$.
The complete off-diagonal row for state $(1,2)$ is
\[
\begin{aligned}
q_{(1,2),3}&=\lambda,&q_{(1,2),(2)}&=\mu_1+\mu_{13},\\
q_{(1,2),(1)}&=\mu_2+\mu_{23},&q_{(1,2),0}&=\mu_{12}+\mu_{123}.
\end{aligned}
\]
A $\{3\}$ shock has no effect. The diagonal is
$-[\lambda+\alpha(\{1,2\})]$, whereas the customer-departure rate is
$\nu_1+\nu_2$, counting each simultaneous departure twice.

A $\{1,2\}$ shock leaves server $3$ busy and has the following outcomes:
\[
\begin{array}{c|c|c}
\text{State before}&\text{State after}&\text{Conditional probability}\\\hline
3&(3)&1\\
4&(1,3)&p_1(\{1,2\})\\
4&(2,3)&p_2(\{1,2\})\\
5&3&1
\end{array}
\]
Different shocks can yield the same transition; for example,
\[
 q_{4,(1,3)}=\mu_{12}p_1(\{1,2\})+\mu_{23}p_3(\{2,3\}).
\]
A single completion from state $4$ instead gives state $3$ immediately.
For a three-server shock, one or two waiting customers give, respectively,
\[
\begin{aligned}
 q_{4,(i)}&=\mu_{123}p_i(\mathcal I),\\
 q_{5,(i,j)}&=\mu_{123}K_{\mathcal I,2}(\{i,j\})
 =\mu_{123}\bigl[p_i(\mathcal I)p_j(\{j,k\})
                 +p_j(\mathcal I)p_i(\{i,k\})\bigr].
\end{aligned}
\]
The last expression sums the two possible assignment orders.
\end{example}

\begin{remark}[More general assignment rules]
\label{rem:assignment-generalization}
A more general time-homogeneous rule may use the current queue state
and fresh independent randomization, allowing joint assignments or
dependence on the number waiting, without anticipating future events. If service remains nonpreemptive
and work-conserving, $Y(t)$ is still Markov. Only boundary transitions
change; the fully busy rates, homogeneous count reduction, and stability
condition remain valid. Boundary probabilities and tail normalization
can change. History-dependent rules may require a larger state, and
tagged-customer delays require the individual assignment order to be
specified. The subsequent analysis uses the fixed sequential rule
\eqref{eq:assignment-probabilities}.
\end{remark}

\subsection{Matrix-analytic formulation of the homogeneous \texorpdfstring{$M/M_D/c$}{M/M(D)/c} model}
\label{sec:homogeneous-MAM}

Assume $\mu_A=\mu_{(|A|)}$. With $m$ busy servers, choosing $\ell$
affected busy servers and $k$ affected idle servers gives the
customer-departure rate
\begin{equation}\label{eq:hom-MAM-rates}
 R_{m,\ell}=\binom{m}{\ell}
   \sum_{k=0}^{c-m}\binom{c-m}{k}\mu_{(\ell+k)},
 \qquad 1\le\ell\le m\le c.
\end{equation}
These rates depend only on the number busy, so $Y(t)$ is strongly
lumpable by customer count for every assignment rule
\eqref{eq:assignment-probabilities}. Thus $X(t)$ has transitions
$n\to n+1$ at rate $\lambda$ and $n\to n-\ell$ at rate
$R_{\min(n,c),\ell}$. Without shock symmetry, $X(t)$ need not be Markov.
Write
\begin{equation}\label{eq:hom-MAM-capacity}
 \Sigma_0=0,\qquad \Sigma_m=\sum_{\ell=1}^{m}R_{m,\ell},\qquad
 \Theta_c=\sum_{\ell=1}^{c}\ell R_{c,\ell},\qquad
 \gamma=\lambda+\Sigma_c.
\end{equation}
Here $\Sigma_m$ counts completion events, whereas $\Theta_c$ counts
customers served per unit time when fully busy. Assume
$\lambda<\Theta_c$, the stability condition of Lemma~\ref{lem:stability}.

\paragraph{Levels and explicit transition blocks.}
For every $m\ge0$, define
\begin{equation}\label{eq:hom-MAM-levels}
 L_m=\{mc,mc+1,\ldots,mc+c-1\},\qquad
 \boldsymbol\pi_m=(\pi_{mc},\pi_{mc+1},\ldots,\pi_{mc+c-1}).
\end{equation}
Within each level, phase $i\in\{1,\ldots,c\}$ represents the customer
count $mc+i-1$. An arrival increases the count by one, and a completion
event decreases it by at most $c$. Thus transitions connect only the same
or adjacent levels, giving a quasi-birth-and-death (QBD) process with
generator
\begin{equation}\label{eq:hom-MAM-Q}
 Q=\bordermatrix{
        & L_0    & L_1    & L_2    & L_3    & \cdots\cr
 L_0    & B_0    & B_1    & 0      & 0      & \cdots\cr
 L_1    & B_{-1} & A_0    & A_1    & 0      & \cdots\cr
 L_2    & 0      & A_{-1} & A_0    & A_1    & \ddots\cr
 L_3    & 0      & 0      & A_{-1} & A_0    & \ddots\cr
 \vdots & \vdots & \vdots & \ddots & \ddots & \ddots\cr
 }.
\end{equation}
All six blocks are $c\times c$, in the phase order of
\eqref{eq:hom-MAM-levels}. \chg{Let $e_i$ denote the $i$th standard basis column vector in $\mathbb R^c$, with a $1$ in position $i$ and zeros elsewhere.} The upward blocks are
\begin{equation}\label{eq:hom-MAM-up}
 B_1=A_1=\lambda e_c e_1^{\top}.
\end{equation}
For $1\le i,j\le c$, the other blocks have explicit entries
\begin{equation}\label{eq:hom-MAM-within}
 (A_0)_{ij}=\begin{cases}
 -\gamma,&j=i,\\
 \lambda,&j=i+1,\\
 R_{c,i-j},&j<i,\\
 0,&j>i+1,
 \end{cases}
\end{equation}
\begin{equation}\label{eq:hom-MAM-down}
 B_{-1}=A_{-1},\qquad
 (A_{-1})_{ij}=\begin{cases}
 R_{c,c+i-j},&j\ge i,\\
 0,&j<i,
 \end{cases}
\end{equation}
\begin{equation}\label{eq:hom-MAM-boundary-block}
 (B_0)_{ij}=\begin{cases}
 -(\lambda+\Sigma_{i-1}),&j=i,\\
 \lambda,&j=i+1,\\
 R_{i-1,i-j},&j<i,\\
 0,&j>i+1.
 \end{cases}
\end{equation}
A downward level transition from phase $i$ to $j$ removes $c+i-j$
customers. Because all states in $L_1$ are fully busy, its downward
block equals the repeating block even when the destination is in $L_0$.
These formulas include vanishing shock rates and satisfy
$(B_0+B_1)\mathbf1=0$ and $(A_{-1}+A_0+A_1)\mathbf1=0$.

\paragraph{The rate matrix and its geometric row.}
Standard matrix-geometric theory~\cite{Neuts:1981} associates with the
repeating blocks the minimal nonnegative solution $R$.
\chg{Here $R$ denotes the QBD rate matrix and should not be confused with the
shock-size rates $R_\ell$ in \eqref{eq:shock-size-rates} or the completion
rates $R_{m,\ell}$ in \eqref{eq:hom-MAM-rates}.}
It solves
\begin{equation}\label{eq:R-quadratic}
 A_1+RA_0+R^2A_{-1}=0.
\end{equation}
Under stability, $\rho(R)<1$, and
$\boldsymbol\pi_{m+1}=\boldsymbol\pi_mR$ for $m\ge1$.
The rank-one structure belongs to a classical explicit-solution class
\cite{Ramaswami-Latouche:1986,vanLeeuwaarden-Winands:2006}. We give the
short rank argument for completeness, then evaluate the geometric row
for this queue. Corollary~\ref{cor:hom-MAM-scalar-tail} below extends
the stationary relation to $m=0$ and gives the scalar geometric tail.

\begin{lemma}\label{lem:hom-MAM-rank-one}
The minimal nonnegative solution of \eqref{eq:R-quadratic} has the form
$R=e_cb$, where $b=(b_1,\ldots,b_c)$ has strictly positive entries.
In particular, $R$ has rank one and $0<b_c=\rho(R)<1$.
\end{lemma}
\begin{proof}
Since $A_0+\gamma I$ is nonnegative, the iteration
\[
 Z_0=0,\qquad
 Z_{k+1}=\frac{1}{\gamma}
 \bigl[A_1+Z_k(A_0+\gamma I)+Z_k^2A_{-1}\bigr]
\]
is monotone and bounded above by $R$. Its limit solves
\eqref{eq:R-quadratic}, so minimality identifies the limit with $R$.
Only the last row of $A_1$ is nonzero; induction shows that the same is
true of every $Z_k$ and hence of $R$. In the last row of the fixed-point
equation, nonnegativity gives
$b_1\ge\lambda/\gamma>0$ and
$b_j\ge\lambda b_{j-1}/\gamma>0$ for $2\le j\le c$.
Finally, the only nonzero eigenvalue of $e_cb$ is $be_c=b_c$, and
$\rho(R)<1$ gives the claim.
\end{proof}

\begin{theorem}[Geometric row of the rate matrix]\label{thm:hom-MAM-geometric-row}
The rate matrix is
\begin{equation}\label{eq:hom-MAM-explicit-R}
 R=e_c(r,r^2,\ldots,r^c),
\end{equation}
where $r\in(0,1)$ is the unique solution in that interval of
\begin{equation}\label{eq:MAM-characteristic}
 \lambda+\sum_{\ell=1}^{c}R_{c,\ell}r^{\ell+1}
       -r(\lambda+\Sigma_c)=0.
\end{equation}
Equivalently, $r$ is the unique solution of
\begin{equation}\label{eq:hom-MAM-root-monotone}
 \lambda=\sum_{\ell=1}^{c}R_{c,\ell}\sum_{k=1}^{\ell}r^k.
\end{equation}
\end{theorem}
\begin{proof}
Write $R=e_cb$, put $q=b_c\in(0,1)$, and set $r=q^{1/c}$.
Then $R^2=qR$. Substituting the explicit blocks into the last row of
\eqref{eq:R-quadratic} gives, for $1\le j\le c$,
\begin{equation}\label{eq:hom-MAM-row-recurrence}
 \gamma b_j=\lambda b_{j-1}
       +\sum_{\ell=1}^{c}R_{c,\ell}b_{j+\ell},
 \qquad b_0=1,\quad b_{c+j}=q b_j\ (1\le j\le c).
\end{equation}
The term $\lambda b_0=\lambda$ comes from $A_1$; terms with
$j+\ell\le c$ come from $bA_0$, and the remaining terms come from
$qbA_{-1}$.

Define $y_j=b_j/r^j$ for $1\le j\le c$ and extend $y_j$
periodically with period $c$. Since $b_c=r^c$, we have $y_0=y_c=1$.
Dividing \eqref{eq:hom-MAM-row-recurrence} by $r^j$ yields
\begin{equation}\label{eq:hom-MAM-cyclic}
 \gamma y_j=\frac{\lambda}{r}y_{j-1}
        +\sum_{\ell=1}^{c}R_{c,\ell}r^\ell y_{j+\ell},
 \qquad 1\le j\le c.
\end{equation}
Summing over one period and using $\sum_{j=1}^{c}y_j>0$ gives
\[
 \gamma=\frac{\lambda}{r}+\sum_{\ell=1}^{c}R_{c,\ell}r^\ell,
\]
which is \eqref{eq:MAM-characteristic}. Consequently,
\eqref{eq:hom-MAM-cyclic}, divided by $\gamma$, expresses each $y_j$
as a convex combination of the periodic values. At an index attaining
their maximum, its predecessor must attain the same maximum because
the coefficient $\lambda/(r\gamma)$ is strictly positive. Iterating
around the cycle shows that all $y_j$ are equal. As $y_c=1$, this proves
$b_j=r^j$ for every $j$.

To establish uniqueness, factor the characteristic polynomial as
\[
 \lambda+\sum_{\ell=1}^{c}R_{c,\ell}r^{\ell+1}
       -r(\lambda+\Sigma_c)
 =(1-r)\left[\lambda-
      \sum_{\ell=1}^{c}R_{c,\ell}\sum_{k=1}^{\ell}r^k\right].
\]
The sum in brackets after $\lambda-$ is a continuous, strictly increasing
function of $r\in[0,1]$, with values $0$ and $\Theta_c$ at the endpoints.
Thus $0<\lambda<\Theta_c$ gives exactly one root in $(0,1)$.
The additional root $r=1$ of \eqref{eq:MAM-characteristic} is excluded
by $\rho(R)<1$.
\end{proof}

\begin{corollary}[Scalar geometric tail]\label{cor:hom-MAM-scalar-tail}
Under the stability condition $\lambda<\Theta_c$, the stationary
customer-count probabilities satisfy
\begin{equation}\label{eq:hom-MAM-scalar-tail}
 \pi_n=\pi_{c-1}r^{\,n-(c-1)},\qquad n\ge c-1,
\end{equation}
where $r\in(0,1)$ is the root specified in
Theorem~\ref{thm:hom-MAM-geometric-row}.
\end{corollary}
\begin{proof}
Put $U=A_0+RA_{-1}$. Standard QBD theory~\cite{Neuts:1981}
ensures that $-U$ is nonsingular. Using
$\boldsymbol\pi_2=\boldsymbol\pi_1R$ in the level-$1$ balance equation
therefore gives
\[
 \boldsymbol\pi_1
 =\boldsymbol\pi_0B_1(-U)^{-1}.
\]
Since $B_1=A_1$ and \eqref{eq:R-quadratic} gives $A_1=R(-U)$,
we obtain $\boldsymbol\pi_1=\boldsymbol\pi_0R$. Thus the
matrix-geometric representation extends to the boundary:
$\boldsymbol\pi_m=\boldsymbol\pi_0R^m$ for all $m\ge0$.

By Theorem~\ref{thm:hom-MAM-geometric-row},
$R=e_c(r,r^2,\ldots,r^c)$ and $R^m=r^{c(m-1)}R$ for $m\ge1$.
Consequently,
\[
 \pi_{mc+j-1}=\pi_{c-1}r^{c(m-1)+j},
 \qquad m\ge1,\quad 1\le j\le c.
\]
Since $c(m-1)+j=(mc+j-1)-(c-1)$, this is
\eqref{eq:hom-MAM-scalar-tail} for every $n\ge c$; the case $n=c-1$
is immediate.
\end{proof}

\paragraph{Finite boundary system and stationary distribution.}
Define the $c\times c$ matrix
\begin{equation}\label{eq:hom-MAM-H}
 H=B_0+RA_{-1}=B_0+e_c(h_1,\ldots,h_c),\qquad
 h_j=\sum_{i=1}^{j}r^iR_{c,c+i-j},\quad 1\le j\le c.
\end{equation}
Thus only the last row of $B_0$ receives a correction. All entries of
$H$ are explicit in the shock rates and the scalar root $r$.

\begin{theorem}[Boundary probabilities and normalization]
\label{thm:hom-MAM-stationary}
$H$ is an irreducible generator on $L_0$. Let
$v=(v_0,\ldots,v_{c-1})$ be its stationary probability vector:
\begin{equation}\label{eq:hom-MAM-finite-system}
 vH=0,\qquad v\mathbf1=1.
\end{equation}
Set
\begin{equation}\label{eq:hom-MAM-normalizing-constant}
 d=v(I-R)^{-1}\mathbf1
   =\sum_{n=0}^{c-2}v_n+\frac{v_{c-1}}{1-r}.
\end{equation}
The unique stationary distribution of $X(t)$ is then
\begin{equation}\label{eq:hom-MAM-stationary-levels}
 \boldsymbol\pi_0=\frac{v}{d},\qquad
 \boldsymbol\pi_m=\boldsymbol\pi_0R^m,\quad m\ge1.
\end{equation}
The scalar tail is given by Corollary~\ref{cor:hom-MAM-scalar-tail},
and the normalization is
\begin{equation}\label{eq:hom-MAM-scalar-normalization}
 \sum_{n=0}^{c-2}\pi_n+\frac{\pi_{c-1}}{1-r}=1.
\end{equation}
\end{theorem}
\begin{proof}
The off-diagonal entries of $H$ are nonnegative. By
\eqref{eq:hom-MAM-root-monotone},
\[
 \sum_{j=1}^{c}h_j
 =\sum_{\ell=1}^{c}R_{c,\ell}\sum_{i=1}^{\ell}r^i
 =\lambda.
\]
Since $B_0\mathbf1=-\lambda e_c$, it follows that $H\mathbf1=0$.
The upward rates between consecutive states in $L_0$ are positive,
and every nonempty state has a positive-rate transition to a smaller
state: $\Sigma_m>0$ for $1\le m<c$. These transitions are retained in $H$,
so $H$ is irreducible. Therefore \eqref{eq:hom-MAM-finite-system} has
a unique strictly positive solution.

For $m\ge1$, $R^m=r^{c(m-1)}R$. Hence
\[
 (I-R)^{-1}=I+\frac{R}{1-r^c},\qquad
 (I-R)^{-1}\mathbf1
 =\bigl(1,\ldots,1,(1-r)^{-1}\bigr)^{\top},
\]
which proves \eqref{eq:hom-MAM-normalizing-constant} and normalizes
\eqref{eq:hom-MAM-stationary-levels}. The level-$0$ balance equation is
\[
 \boldsymbol\pi_0 B_0+\boldsymbol\pi_1B_{-1}
 =\boldsymbol\pi_0H=0.
\]
For every $m\ge1$, including $m=1$ because $B_1=A_1$, the remaining
level balances follow from the single matrix identity
\[
 \boldsymbol\pi_{m-1}A_1+\boldsymbol\pi_mA_0
       +\boldsymbol\pi_{m+1}A_{-1}
 =\boldsymbol\pi_0R^{m-1}
       (A_1+RA_0+R^2A_{-1})=0.
\]
Thus the constructed probability vector is stationary. Irreducibility
of $X(t)$ gives uniqueness.
\end{proof}

\begin{corollary}[Boundary recursion]
\label{cor:hom-MAM-boundary-recursion}
Under the assumptions of Theorem~\ref{thm:hom-MAM-stationary}, set
$w_{c-1}=1$ and compute successively, in decreasing order of $n$,
\begin{equation}\label{eq:hom-MAM-boundary-recursion}
 w_n=\frac{1}{\lambda}
 \left[\sum_{j=n+1}^{c-1}w_j\sum_{\ell=j-n}^{j}R_{j,\ell}
       +\sum_{k=1}^{n+1}h_k\right],
 \qquad n=c-2,\ldots,0,
\end{equation}
where $R_{j,\ell}$ and $h_k$ are given in
\eqref{eq:hom-MAM-rates} and \eqref{eq:hom-MAM-H}, respectively.
Define
\begin{equation}\label{eq:hom-MAM-recursive-normalizer}
 Z=\sum_{n=0}^{c-2}w_n+\frac{1}{1-r}.
\end{equation}
Then the stationary probabilities are
\begin{equation}\label{eq:hom-MAM-recursive-probabilities}
 \pi_n=\begin{cases}
 w_n/Z,&0\le n\le c-2,\\[1mm]
 r^{\,n-c+1}/Z,&n\ge c-1.
 \end{cases}
\end{equation}
In particular, $w_n=\pi_n/\pi_{c-1}$ for $0\le n\le c-1$.
For each fixed $c$, this finite recursion expresses the boundary
probabilities as rational functions of $r$ and the transition rates.
\chg{Together with the scalar geometric tail, all stationary probabilities
are therefore expressible from the single root $r$ and the model rates.}
\end{corollary}
\begin{proof}
Let $v$ be the strictly positive stationary vector of $H$ from
Theorem~\ref{thm:hom-MAM-stationary}. Fix $0\le n\le c-2$.
The only transition of $H$ leaving $\{0,\ldots,n\}$ is the arrival
from $n$ to $n+1$, at rate $\lambda$. A completion from state $j>n$
enters this set precisely when its batch size is at least $j-n$.
The last row of $H$ also contributes the corrections $h_1,\ldots,h_{n+1}$.
Equating stationary flow in the two directions gives
\[
 \lambda v_n
 =\sum_{j=n+1}^{c-1}v_j\sum_{\ell=j-n}^{j}R_{j,\ell}
       +v_{c-1}\sum_{k=1}^{n+1}h_k.
\]
Dividing by $v_{c-1}>0$ gives
\eqref{eq:hom-MAM-boundary-recursion}. Its right-hand side involves
only ratios already computed, beginning with $w_{c-1}=1$.
Theorem~\ref{thm:hom-MAM-stationary} gives
$v_n/v_{c-1}=\pi_n/\pi_{c-1}$. The scalar tail in
Corollary~\ref{cor:hom-MAM-scalar-tail} and the normalization
\eqref{eq:hom-MAM-scalar-normalization} now yield
\eqref{eq:hom-MAM-recursive-normalizer}--\eqref{eq:hom-MAM-recursive-probabilities}.
\end{proof}

For computation, solve the strictly monotone scalar equation
\eqref{eq:hom-MAM-root-monotone}, for example by bisection on $(0,1)$,
then use the backward recursion \eqref{eq:hom-MAM-boundary-recursion}
and normalization \eqref{eq:hom-MAM-recursive-normalizer}.
Alternatively, solve \eqref{eq:hom-MAM-finite-system} for the stationary
vector of $H$ in \eqref{eq:hom-MAM-H}. The QBD identity satisfies the
boundary and tail equations together, avoiding the separate consistency
check used for $M/M_D/2$ in Thapa and Zhao~\cite{Thapa-Zhao:2026}
and for general $c$ in the preliminary version of this
work~\cite{Bu-Thapa-Zhao:2026}.

\begin{example}[Explicit solution for three homogeneous servers]
\label{exa:hom-MAM-three-servers}
\normalfont
Let $c=3$ and write $a=\mu_{(1)}$, $b=\mu_{(2)}$, and $u=\mu_{(3)}$.
The rates in \eqref{eq:hom-MAM-rates} are
\[
\begin{array}{c|ccc}
m&R_{m,1}&R_{m,2}&R_{m,3}\\\hline
1&a+2b+u&0&0\\
2&2(a+b)&b+u&0\\
3&3a&3b&u
\end{array}
\]
Under $\lambda<3a+6b+3u$, the root satisfies
$\lambda=3ar+3b(r+r^2)+u(r+r^2+r^3)$.
Corollary~\ref{cor:hom-MAM-boundary-recursion} gives
\[
 w_2=1,\qquad
 w_1=\frac{2a+3b(1+r)+u(1+r+r^2)}{\lambda},\qquad
 w_0=\frac{(a+2b+u)w_1+b+u(1+r)}{\lambda}.
\]
With $Z=w_0+w_1+(1-r)^{-1}$,
\[
 \pi_0=\frac{w_0}{Z},\qquad \pi_1=\frac{w_1}{Z},\qquad
 \pi_n=\frac{r^{n-2}}{Z},\quad n\ge2.
\]
The three-server blocks are also displayed for the general
heterogeneous case in Example~\ref{exa:hetero-MAM-three}.
\end{example}

\begin{remark}[Independent-service limit]\label{rem:hom-MAM-independent}
For $\mu_{(1)}=\mu>0$ and $\mu_{(\ell)}=0$ when $\ell\ge2$,
$R_{m,1}=m\mu$, all larger batches vanish, and $r=\lambda/(c\mu)$.
The boundary recursion reduces to
$w_n=(n+1)\mu w_{n+1}/\lambda$, giving
\[
 w_n=\frac{(c-1)!}{n!}\left(\frac{\mu}{\lambda}\right)^{c-1-n},
 \qquad 0\le n\le c-1.
\]
Thus $\pi_n=\pi_0(\lambda/\mu)^n/n!$ below $c$, followed by the
classical $M/M/c$ geometric tail. This limit is included although
$R_{c,c}=0$, since no inverse of $A_{-1}$ was used.
\end{remark}

\subsection{Matrix-analytic formulation of the heterogeneous \texorpdfstring{$M/M_D/c$}{M/M(D)/c} model}
\label{sec:heterogeneous-MAM}

We retain the shock rates and random-assignment rule of Section~\ref{subsec:shock-construction},
without imposing symmetry across servers. \chg{Recall the shock-size rates $R_\ell$
defined in \eqref{eq:shock-size-rates}.} Write
\begin{equation}\label{eq:hetero-MAM-rates}
 \Sigma=\sum_{\ell=1}^cR_\ell,\qquad
 \Theta=\sum_{\ell=1}^c\ell R_\ell,\qquad
 \gamma=\lambda+\Sigma.
\end{equation}
Throughout this subsection, $\lambda<\Theta$. By
Lemma~\ref{lem:stability}, the queue has a unique stationary distribution
on the class $S_0$ accessible from $0$; probabilities of inaccessible
states are zero. The customer count alone is generally not Markov below
$c$, so the boundary must distinguish the busy-server subsets.

\paragraph{Levels and explicit transition blocks.}
Identify the empty subset with state $0$ and set
\begin{equation}\label{eq:hetero-MAM-levels}
 L_0=\{s\subsetneq\mathcal I\},\qquad
 L_m=\{mc,mc+1,\ldots,mc+c-1\},\quad m\ge1.
\end{equation}
Order the $d_0=2^c-1$ states of $L_0$ first by cardinality and then
lexicographically. In every $L_m$, $m\ge1$, phase $j$ represents
$mc+j-1$ customers. Let $\boldsymbol\pi_0=(\pi_s:s\in L_0)$ and
$\boldsymbol\pi_m=(\pi_{mc},\ldots,\pi_{mc+c-1})$ for $m\ge1$.
The generator is
\begin{equation}\label{eq:hetero-MAM-Q}
 Q=\bordermatrix{
        & L_0    & L_1    & L_2    & L_3    & \cdots\cr
 L_0    & B_0    & B_1    & 0      & 0      & \cdots\cr
 L_1    & B_{-1} & A_0    & A_1    & 0      & \cdots\cr
 L_2    & 0      & A_{-1} & A_0    & A_1    & \ddots\cr
 L_3    & 0      & 0      & A_{-1} & A_0    & \ddots\cr
 \vdots & \vdots & \vdots & \ddots & \ddots & \ddots\cr
 }.
\end{equation}
Here $B_0$, $B_1$, and $B_{-1}$ have dimensions $d_0\times d_0$,
$d_0\times c$, and $c\times d_0$, respectively, whereas the repeating
blocks are $c\times c$. All six blocks are specified below.

\chg{Let $e_j$ denote the $j$th standard basis column vector in $\mathbb R^c$, with a $1$ in position $j$ and zeros elsewhere.} For
$1\le i,j\le c$, the repeating blocks are
\begin{align}
 A_1&=\lambda e_ce_1^{\top},
 \label{eq:hetero-MAM-up}\\
 (A_0)_{ij}&=
 \begin{cases}
 -\gamma,&j=i,\\
 \lambda,&j=i+1,\\
 R_{i-j},&j<i,\\
 0,&j>i+1,
 \end{cases}
 \label{eq:hetero-MAM-within}\\
 (A_{-1})_{ij}&=
 \begin{cases}
 R_{c+i-j},&j\ge i,\\
 0,&j<i.
 \end{cases}
 \label{eq:hetero-MAM-down}
\end{align}
Thus these blocks have the same form as in Section~\ref{sec:homogeneous-MAM}, with
$R_{c,\ell}$ replaced by $R_\ell$. The upward transition crosses from
the last phase of one level to the first phase of the next.

For proper subsets $s,t\subsetneq\mathcal I$, the within-boundary
block is
\begin{equation}\label{eq:hetero-MAM-B0}
 (B_0)_{s,t}=
 \begin{cases}
 -[\lambda+\alpha(s)],&t=s,\\
 \lambda p_i(\mathcal I\setminus s),
       &t=s\cup\{i\},\quad |s|\le c-2,\\
 \mu_{s\setminus t}^{s},&t\subsetneq s,\\
 0,&\text{otherwise}.
 \end{cases}
\end{equation}
The marginal rates $\mu_D^s$ and $\alpha(s)$ are those of Section~\ref{subsec:shock-construction}.
Define the column vector $\chi$ indexed by $L_0$ by
$\chi_s=\mathbf1_{\{|s|=c-1\}}$. An arrival leaves $L_0$ precisely
from a state with $c-1$ busy servers, and always enters state $c$.
Consequently,
\begin{equation}\label{eq:hetero-MAM-B1}
 B_1=\lambda\chi e_1^{\top},\qquad
 (B_1)_{s,j}=\lambda\mathbf1_{\{|s|=c-1,\ j=1\}}.
\end{equation}
Finally, the row of $B_{-1}$ indexed by $j$ corresponds to state
$c+j-1$. Using the assignment probabilities
$K_{A,j-1}$ defined in \eqref{eq:assignment-subset-probability},
\begin{equation}\label{eq:hetero-MAM-Bminus}
 (B_{-1})_{j,s}=
 \sum_{\substack{A\in\mathcal P:\ \mathcal I\setminus A\subseteq s\\
                         |A|=c+j-1-|s|}}
       \mu_A K_{A,j-1}(s\cap A),
 \qquad 1\le j\le c,\quad s\in L_0.
\end{equation}
The $j-1$ waiting customers refill $s\cap A$ after the shock removes
$c+j-1-|s|$ customers; empty sums are zero. The block row sums are
\begin{equation}\label{eq:hetero-MAM-block-sums}
 B_0\mathbf1=-\lambda\chi,\qquad
 (B_{-1}\mathbf1)_j=\sum_{\ell=j}^c R_\ell,\qquad
 (A_{-1}+A_0+A_1)\mathbf1=0.
\end{equation}
Here and below, each vector $\mathbf1$ has the dimension required by
its matrix product. Only the boundary blocks depend on the assignment
probabilities or on the identities within a shock set.
\begin{chgblock}
In particular, the repeating blocks $A_{-1},A_0,A_1$ are independent of
the idle-server assignment rule. In the heterogeneous model the full
generator $Q$ can nevertheless depend on that rule through the boundary
blocks $B_0,B_1,B_{-1}$.
\end{chgblock}

\paragraph{Rate matrix and the connection to the boundary.}
\begin{lemma}\label{lem:hetero-MAM-rate}
The minimal nonnegative solution of
$A_1+RA_0+R^2A_{-1}=0$ is
\begin{equation}\label{eq:hetero-MAM-R}
 R=e_c b,\qquad b=(r,r^2,\ldots,r^c),
\end{equation}
where $r\in(0,1)$ is the unique solution of
\begin{equation}\label{eq:hetero-MAM-root}
 \lambda=\sum_{\ell=1}^cR_\ell\sum_{k=1}^{\ell}r^k.
\end{equation}
Equivalently,
\begin{equation}\label{eq:hetero-MAM-characteristic}
 \lambda+\sum_{\ell=1}^cR_\ell r^{\ell+1}
       -r(\lambda+\Sigma)=0.
\end{equation}
Put $U=A_0+RA_{-1}$. Then $-U$ is nonsingular, and the matrix connecting
the boundary to the first fully busy level is
\begin{equation}\label{eq:hetero-MAM-F}
 F:=B_1(-U)^{-1}=\chi b.
\end{equation}
The stationary level vectors satisfy
\begin{equation}\label{eq:hetero-MAM-level-solution}
 \boldsymbol\pi_1=\boldsymbol\pi_0F,\qquad
 \boldsymbol\pi_m=\boldsymbol\pi_0FR^{m-1},\quad m\ge1.
\end{equation}
\end{lemma}
\begin{proof}
The repeating blocks have exactly the form used in
Lemma~\ref{lem:hom-MAM-rank-one} and
Theorem~\ref{thm:hom-MAM-geometric-row}. Their proofs apply with
$R_{c,\ell}$ replaced by $R_\ell$, giving
\eqref{eq:hetero-MAM-R}. The right-hand side of
\eqref{eq:hetero-MAM-root} is continuous and strictly increasing on
$[0,1]$, with endpoint values $0$ and $\Theta$. Hence it has exactly
one solution in $(0,1)$ under $0<\lambda<\Theta$.

Standard QBD theory~\cite{Neuts:1981} gives the nonsingularity of
$-U$ and $\boldsymbol\pi_{m+1}=\boldsymbol\pi_mR$ for $m\ge1$.
The rate-matrix equation gives $R(-U)=A_1$. Since
$B_1=\chi e_c^{\top}A_1$,
\[
 B_1(-U)^{-1}=\chi e_c^{\top}R=\chi b.
\]
The level-$1$ balance equation, with
$\boldsymbol\pi_2=\boldsymbol\pi_1R$, is
$\boldsymbol\pi_0B_1+\boldsymbol\pi_1U=0$.
Solving it gives \eqref{eq:hetero-MAM-level-solution}.
\end{proof}

\paragraph{Finite boundary generator and stationary probabilities.}
Define the row vector $h=bB_{-1}$ and the $d_0\times d_0$ matrix
\begin{equation}\label{eq:hetero-MAM-H}
 H=B_0+FB_{-1}=B_0+\chi h,\qquad
 h_s=\sum_{j=1}^c r^j(B_{-1})_{j,s}.
\end{equation}
All entries are explicit in the shock rates, assignment probabilities,
and scalar root $r$. Only rows corresponding to $c-1$ busy servers
receive the correction $h$.

\begin{theorem}[Boundary probabilities and scalar geometric tail]
\label{thm:hetero-MAM-stationary}
The matrix $H$ is a finite conservative generator with a unique closed
communicating class, $L_0\cap S_0$, \chg{where $S_0$ denotes the communicating
class reachable from state $0$; see Section~\ref{sec:queue-model}.} Its unique stationary probability
vector $v$ is determined by
\begin{equation}\label{eq:hetero-MAM-finite-system}
 vH=0,\qquad v\mathbf1=1.
\end{equation}
Set
\begin{equation}\label{eq:hetero-MAM-normalizer}
 \eta=1+\frac{r}{1-r}\,v\chi.
\end{equation}
Then the stationary probabilities of the queue are
\begin{equation}\label{eq:hetero-MAM-probabilities}
 \pi_s=\frac{v_s}{\eta},\quad s\in L_0,\qquad
 \pi_n=\frac{v\chi}{\eta}\,r^{\,n-c+1},\quad n\ge c.
\end{equation}
In particular,
\begin{equation}\label{eq:hetero-MAM-scalar-tail}
 \pi_c=r\sum_{|s|=c-1}\pi_s,\qquad
 \pi_n=\pi_c r^{\,n-c},\quad n\ge c,
\end{equation}
and the normalization is
\begin{equation}\label{eq:hetero-MAM-normalization}
 \sum_{s\in L_0}\pi_s+\frac{\pi_c}{1-r}=1.
\end{equation}
The probabilities are strictly positive on $S_0$ and zero outside it.
If $S_0=S$, then $H$ is irreducible and every component of $v$ is
strictly positive.
\end{theorem}
\begin{proof}
Off-diagonal entries of $H$ are nonnegative. By
\eqref{eq:hetero-MAM-block-sums} and \eqref{eq:hetero-MAM-root},
\[
 h\mathbf1=\sum_{j=1}^c r^j\sum_{\ell=j}^cR_\ell
             =\sum_{\ell=1}^cR_\ell\sum_{j=1}^{\ell}r^j
             =\lambda.
\]
Thus $H\mathbf1=-\lambda\chi+\lambda\chi=0$.

Every state of $L_0$ can reach $0$ through positive-rate shocks with no
intervening arrivals, as in the proof of Lemma~\ref{lem:stability}.
These transitions are retained in $B_0$. Paths within $L_0$ are also
retained. A path that leaves $L_0$ must first enter state $c$ from a
state with $c-1$ busy servers. It can return to $L_0$ at $t$ precisely
when $(B_{-1})_{j,t}>0$ for some $j$: every state $c+j-1$ is reachable
from $c$ by arrivals, and a return from the fully busy region can occur
only from one of these $c$ states. Since every $r^j$ is positive,
$h_t>0$ is equivalent to this return being possible. The correction
$\chi h$ therefore preserves reachability between boundary states
through such excursions. Consequently, the states reachable from $0$
under $H$ are exactly $L_0\cap S_0$; they form the unique closed
communicating class, and every other boundary state is transient.
This proves the assertions about $v$.

Let $\boldsymbol\pi_0=v/\eta$ and use
\eqref{eq:hetero-MAM-level-solution} for the remaining levels.
Since $bR=r^cb$, we have
$FR^{m-1}=\chi r^{c(m-1)}b$. Its $j$th column yields
\eqref{eq:hetero-MAM-probabilities} at $n=mc+j-1$.
The total mass in levels $m\ge1$ is
\[
 \frac{v\chi}{\eta}\sum_{k=1}^{\infty}r^k
       =\frac{v\chi}{\eta}\frac{r}{1-r},
\]
so \eqref{eq:hetero-MAM-normalizer} gives total mass one.

The level-$0$ balance equation is
\[
 \boldsymbol\pi_0B_0+\boldsymbol\pi_1B_{-1}
      =\boldsymbol\pi_0H=0.
\]
At level $1$, $F(-U)=B_1$ gives
\[
 \boldsymbol\pi_0B_1+\boldsymbol\pi_1A_0
       +\boldsymbol\pi_2A_{-1}
      =\boldsymbol\pi_0(B_1+FU)=0.
\]
For $m\ge2$, the remaining equations follow from
\[
 \boldsymbol\pi_{m-1}A_1+\boldsymbol\pi_mA_0
       +\boldsymbol\pi_{m+1}A_{-1}
 =\boldsymbol\pi_0FR^{m-2}
       (A_1+RA_0+R^2A_{-1})=0.
\]
Stationarity follows. Lemma~\ref{lem:stability} gives uniqueness and
positivity on $S_0$.
\end{proof}

\begin{corollary}[Finite matrix representation]
\label{cor:hetero-MAM-matrix-representation}
Under the assumptions of Theorem~\ref{thm:hetero-MAM-stationary},
$-B_0$ is nonsingular and its inverse is nonnegative. Define
\begin{equation}\label{eq:hetero-MAM-occupation-vector}
 w=h(-B_0)^{-1},\qquad
 Z=w\mathbf1+\frac{r}{1-r}.
\end{equation}
Then $w\chi=1$, and the stationary probabilities are
\begin{equation}\label{eq:hetero-MAM-matrix-probabilities}
 \pi_s=\frac{w_s}{Z},\quad s\in L_0,\qquad
 \pi_n=\frac{r^{\,n-c+1}}{Z},\quad n\ge c.
\end{equation}
In particular, $v=w/(w\mathbf1)$ is the stationary probability vector
of $H$ in \eqref{eq:hetero-MAM-finite-system}.
\end{corollary}
\begin{proof}
$B_0$ is the generator of the boundary process killed on entering
state $c$. From every proper subset, at most $c$ consecutive arrivals
without an intervening shock take the process out of $L_0$. Since
$L_0$ is finite and $\lambda>0$, absorption occurs with finite mean
time. Thus $B_0$ is transient and
\[
 (-B_0)^{-1}=\int_0^\infty e^{B_0t}\,dt\ge0.
\]
Using $(-B_0)\mathbf1=\lambda\chi$ and $h\mathbf1=\lambda$ gives
\[
 w\chi=\frac{1}{\lambda}h(-B_0)^{-1}(-B_0)\mathbf1=1.
\]
Consequently, $wH=wB_0+(w\chi)h=-h+h=0$. Uniqueness in
Theorem~\ref{thm:hetero-MAM-stationary} identifies
$v=w/(w\mathbf1)$. Substituting this expression into
\eqref{eq:hetero-MAM-normalizer}--\eqref{eq:hetero-MAM-probabilities}
gives \eqref{eq:hetero-MAM-matrix-probabilities}. The total boundary
mass is $w\mathbf1/Z$, while the mass in states $n\ge c$ is
$r/[(1-r)Z]$, proving the stated normalization.
\end{proof}

\begin{remark}[Computation in the heterogeneous model]
\begin{chgblock}
Compute $r$ from \eqref{eq:hetero-MAM-root}, then solve the finite boundary
system $w(-B_0)=h$ and normalize using
\eqref{eq:hetero-MAM-occupation-vector}; equivalently, solve the
$(2^c-1)$-state stationary system for $H$. Thus the stationary distribution
is obtained from a single scalar root followed by a finite-dimensional
boundary solve. Inaccessible states receive zero mass. Different completion
rates across busy subsets prevent the homogeneous scalar recursion from
closing by customer count. The tail ratio depends only on $\lambda$ and
$R_\ell$, while assignment can affect boundary weights and tail coefficients.
The QBD identities again satisfy boundary and tail equations together,
without requiring $A_{-1}$ to be invertible.
\end{chgblock}
\end{remark}

\paragraph{An explicit three-server example.}
\begin{example}[Heterogeneous $M/M_D/3$]
\label{exa:hetero-MAM-three}
\normalfont
Use the notation of Example~\ref{exa:stability-3}, with
\[
 L_0=\{0,(1),(2),(3),(1,2),(1,3),(2,3)\},\qquad L_1=\{3,4,5\}.
\]
Write $u=\mu_{123}$,
$a_i=p_i(\{1,2,3\})$, and $p_i^{ij}=p_i(\{i,j\})$.
Thus $\sum_i a_i=1$ and $p_i^{ij}+p_j^{ij}=1$.
For $\{i,j,k\}=\{1,2,3\}$, put
\[
 \alpha_{ij}=\Sigma-\mu_k,\qquad
 t_{ij}=\mu_{ik}p_i^{ik}+\mu_{jk}p_j^{jk},\qquad
 \theta_{ij}=a_i p_j^{jk}+a_j p_i^{ik}.
\]
Here $t_{ij}$ is the total rate from state $4$ to $(i,j)$ after a
two-server shock, and $\theta_{ij}=K_{\mathcal I,2}(\{i,j\})$ is the
probability of refilling servers $i,j$ after a three-server shock from
state $5$.

The repeating rates and blocks are
\[
 R_1=\mu_1+\mu_2+\mu_3,\qquad
 R_2=\mu_{12}+\mu_{13}+\mu_{23},\qquad R_3=u,
\]
\begin{equation}\label{eq:hetero-MAM-c3-repeating}
 \begin{split}
 A_0&=\begin{pmatrix}-\gamma&\lambda&0\\
 R_1&-\gamma&\lambda\\R_2&R_1&-\gamma\end{pmatrix},\qquad
 A_1=\lambda\begin{pmatrix}0&0&0\\0&0&0\\1&0&0\end{pmatrix},\\[1mm]
 A_{-1}&=\begin{pmatrix}R_3&R_2&R_1\\0&R_3&R_2\\0&0&R_3\end{pmatrix}.
 \end{split}
\end{equation}
To display $B_0$ compactly, group its rows and columns as the empty
state, the three singleton states, and the three pair states. Define
\[
 \begin{gathered}
 a=(a_1,a_2,a_3),\qquad \nu=(\nu_1,\nu_2,\nu_3)^{\top},\qquad
 g=(\mu_{12}+u,\mu_{13}+u,\mu_{23}+u)^{\top},\\
 D_1=\operatorname{diag}(\lambda+\nu_1,\lambda+\nu_2,\lambda+\nu_3),\\
 D_2=\operatorname{diag}(\lambda+\alpha_{12},
                         \lambda+\alpha_{13},\lambda+\alpha_{23}).
 \end{gathered}
\]
Then
\begin{equation}\label{eq:hetero-MAM-c3-B0}
 B_0=\begin{pmatrix}
 -\lambda&\lambda a&0_{1\times3}\\
 \nu&-D_1&\lambda P\\
 g&C&-D_2
 \end{pmatrix},
\end{equation}
where the $3\times3$ blocks are explicitly
\[
 P=\begin{pmatrix}
 p_2^{23}&p_3^{23}&0\\
 p_1^{13}&0&p_3^{13}\\
 0&p_1^{12}&p_2^{12}
 \end{pmatrix},\qquad
 C=\begin{pmatrix}
 \mu_2+\mu_{23}&\mu_1+\mu_{13}&0\\
 \mu_3+\mu_{23}&0&\mu_1+\mu_{12}\\
 0&\mu_3+\mu_{13}&\mu_2+\mu_{12}
 \end{pmatrix}.
\]
The two connecting blocks are
\begin{equation}\label{eq:hetero-MAM-c3-connecting}
 \begin{split}
 B_1&=\lambda(0,0,0,0,1,1,1)^{\top}(1,0,0),\\[1mm]
 B_{-1}&=\begin{pmatrix}
 u&\mu_{23}&\mu_{13}&\mu_{12}&\mu_3&\mu_2&\mu_1\\
 0&u a_1&u a_2&u a_3&t_{12}&t_{13}&t_{23}\\
 0&0&0&0&u\theta_{12}&u\theta_{13}&u\theta_{23}
 \end{pmatrix}.
 \end{split}
\end{equation}
The rows of $B_{-1}$ correspond to states $3,4,5$, with zero, one,
and two waiting customers, respectively.

Under $\lambda<R_1+2R_2+3R_3$, the root is determined by
\[
 \lambda=R_1r+R_2(r+r^2)+R_3(r+r^2+r^3).
\]
With $\chi=(0,0,0,0,1,1,1)^{\top}$, the boundary generator is
\begin{equation}\label{eq:hetero-MAM-c3-H}
 \begin{split}
 H&=B_0+\chi(h_0,h_{(1)},h_{(2)},h_{(3)},
                    h_{(1,2)},h_{(1,3)},h_{(2,3)}),\\
 h_0&=ru,\qquad
 h_{(i)}=r\mu_{jk}+r^2u a_i,\\
 h_{(i,j)}&=r\mu_k+r^2t_{ij}+r^3u\theta_{ij},
 \qquad \{i,j,k\}=\{1,2,3\}.
 \end{split}
\end{equation}
Solving this seven-state stationary system and applying
Theorem~\ref{thm:hetero-MAM-stationary} gives all queue-state
probabilities, including $\pi_n=\pi_3r^{n-3}$ for $n\ge3$.

Assignment effects on stationary performance are compared in
Section~\ref{subsec:numerical-assignment}.
\end{example}

\paragraph{Special cases and reductions.}
\begin{remark}[Reduction to homogeneous servers]
\label{rem:hetero-MAM-homogeneous}
When $\mu_A=\mu_{(|A|)}$, $R_\ell=R_{c,\ell}$ and the scalar-root
equations coincide. Aggregating the boundary by cardinality gives
\[
 \bar\pi_m=\sum_{|s|=m}\pi_s,\qquad 0\le m<c,
\]
and the generator \eqref{eq:hom-MAM-H}, with completion rates
\eqref{eq:hom-MAM-rates}. Thus Theorem~\ref{thm:hom-MAM-stationary}
is recovered for every fixed idle-set assignment rule. Individual
subsets of the same size need not have equal probabilities.
\end{remark}

\begin{example}[Recovery of the $M/M_D/2$ model]
\label{exa:hetero-MAM-two}
\normalfont
Let $c=2$, write $u=\mu_{12}$, and set
$p=p_1(\{1,2\})$, $q=1-p$. Then
$R_1=\mu_1+\mu_2$, $R_2=u$, and stability is
$\lambda<\mu_1+\mu_2+2u$. The scalar equation reduces to
\begin{equation}\label{eq:hetero-MAM-two-root}
 ur^2+(\mu_1+\mu_2+u)r-\lambda=0,
\end{equation}
whose root can be written without division by $u$ as
\[
 r=\frac{2\lambda}{\mu_1+\mu_2+u+
       \sqrt{(\mu_1+\mu_2+u)^2+4u\lambda}}.
\]
On $L_0=\{0,(1),(2)\}$, define
\[
 h=(ru,\ r\mu_2+r^2up,\ r\mu_1+r^2uq).
\]
The explicit boundary generator is
\begin{equation}\label{eq:hetero-MAM-two-H}
 H=\begin{pmatrix}
 -\lambda&\lambda p&\lambda q\\
 \mu_1+u&-(\lambda+\mu_1+u)&0\\
 \mu_2+u&0&-(\lambda+\mu_2+u)
 \end{pmatrix}
 +(0,1,1)^{\top}h.
\end{equation}
The scalar tail satisfies
$\pi_2=r(\pi_{(1)}+\pi_{(2)})$ and
$\pi_n=\pi_2r^{n-2}$ for $n\ge2$.
With $(1)\leftrightarrow(1,0)$ and $(2)\leftrightarrow(0,1)$, these
rates coincide with \cite{Thapa-Zhao:2026} for general $p$, including
the same assignment rule at arrivals and after simultaneous completions.
Uniqueness gives the same stationary distribution, now directly from
the QBD identities.
\end{example}

\begin{corollary}[Independent heterogeneous servers]
\label{cor:hetero-MAM-independent}
Suppose $\mu_A=0$ whenever $|A|\ge2$, and write
$M=\sum_{i=1}^c\mu_i$. The resulting $M/M_i/c$ queue retains the
assignment probabilities $p_i(E)$. It is stable if and only if
$\lambda<M$, and its tail ratio is $\rho=\lambda/M$.
For $t\in L_0$, let
\[
 g_t=\begin{cases}
 \mu_j,&\mathcal I\setminus t=\{j\},\\
 0,&|t|<c-1.
 \end{cases}
\]
The boundary generator reduces to
\begin{equation}\label{eq:hetero-MAM-independent-H}
 H=B_0+\rho\chi g,
\end{equation}
where $B_0$ is given by \eqref{eq:hetero-MAM-B0} with
$\alpha(s)=\sum_{i\in s}\mu_i$ and only single-server downward
transitions, $(B_0)_{s,s\setminus\{i\}}=\mu_i$.
The boundary probabilities follow from
\eqref{eq:hetero-MAM-finite-system}--\eqref{eq:hetero-MAM-probabilities}
with $r=\rho$, and
\begin{equation}\label{eq:hetero-MAM-independent-tail}
 \pi_c=\rho\sum_{|s|=c-1}\pi_s,\qquad
 \pi_n=\pi_c\rho^{n-c},\quad n\ge c,\qquad
 \sum_{s\in L_0}\pi_s+\frac{\pi_c}{1-\rho}=1.
\end{equation}
\end{corollary}
\begin{proof}
Only singleton shocks remain, so $R_1=M$ and $R_\ell=0$ for
$\ell\ge2$. Equation~\eqref{eq:hetero-MAM-root} gives
$r=\lambda/M$. A single completion can enter $L_0$ from $L_1$ only
from state $c$; consequently $B_{-1}=e_1g$ and
$bB_{-1}=\rho g$. The stated formulas now follow from
Theorem~\ref{thm:hetero-MAM-stationary}.
\end{proof}

If additionally $\mu_i=\mu$ for every $i$, aggregation by customer
count recovers the classical $M/M/c$ solution with
$\rho=\lambda/(c\mu)$, as in Remark~\ref{rem:hom-MAM-independent}.

\begin{corollary}[Explicit independent two-server solution]
\label{cor:hetero-MAM-independent-two}
Let $c=2$, $u=\mu_{12}=0$, $p=p_1(\{1,2\})$, and
$\rho=\lambda/(\mu_1+\mu_2)<1$. Then
\begin{align}
 \pi_{(1)}&=\frac{p+\rho}{1+2\rho}\frac{\lambda}{\mu_1}\pi_0,
 &\pi_{(2)}&=\frac{1-p+\rho}{1+2\rho}\frac{\lambda}{\mu_2}\pi_0,
 \label{eq:hetero-MAM-independent-two-boundary}\\
 \pi_2&=\rho(\pi_{(1)}+\pi_{(2)}),
 &\pi_n&=\pi_2\rho^{n-2},\quad n\ge2,
 \label{eq:hetero-MAM-independent-two-tail}
\end{align}
where
\begin{equation}\label{eq:hetero-MAM-independent-two-empty}
 \pi_0=\left[
 1+\frac{\lambda}{(1-\rho)(1+2\rho)}
       \left(\frac{p+\rho}{\mu_1}
                  +\frac{1-p+\rho}{\mu_2}\right)
 \right]^{-1}.
\end{equation}
\end{corollary}
\begin{proof}
The preceding corollary gives the tail. Substitute
$\pi_2=\rho(\pi_{(1)}+\pi_{(2)})$ into the two singleton-state
balance equations and use $\lambda=\rho(\mu_1+\mu_2)$. Writing
$x=\mu_1\pi_{(1)}$ and $y=\mu_2\pi_{(2)}$ gives
\[
 (1+\rho)x-\rho y=\lambda p\pi_0,\qquad
 -\rho x+(1+\rho)y=\lambda(1-p)\pi_0.
\]
The determinant is $1+2\rho$, so solving gives
\eqref{eq:hetero-MAM-independent-two-boundary}. Normalization is
$\pi_0+(\pi_{(1)}+\pi_{(2)})/(1-\rho)=1$, which yields
\eqref{eq:hetero-MAM-independent-two-empty}.
\end{proof}

Setting $p=1$ gives the server-$1$-first rule. Labelling server $1$
as the faster server recovers Singh~\cite{Singh:1970} with $\beta=1$
(no balking), retaining the literature comparison within the
random-assignment model.

\section{A Markovian characterization of the \texorpdfstring{$M/M_D/c$}{M/M(D)/c} service mechanism}
\label{sec:Markovian-characterization}

Retain the queue architecture and assignment rule of
Section~\ref{sec:queue-model}. We characterize when a general Markovian
completion mechanism admits the Poisson shock construction, without
assuming stability or Poisson service shocks at the outset.

\paragraph{Markovian completion rates and consistency.}
For each nonempty busy set $B\subseteq\mathcal I$ and nonempty
$D\subseteq B$, let $\delta_D(B)\ge0$ be the finite rate at which
exactly the services at the servers in $D$ complete simultaneously.
These are rates conditional on the full observed history of arrivals,
labelled service completions, and assignment choices. They depend only
on $B$, and not on service ages, the waiting population, or earlier
events. At each completion, the waiting customers are assigned by the
same independent marks and sequential rule as in Section~\ref{subsec:shock-construction}.
Arrivals remain an independent Poisson process of rate $\lambda$.
We assume $\delta_{\{i\}}(\{i\})>0$ for each server $i$.

The finitely many finite rates and the assignment kernels define a
unique nonexplosive Markov chain on \eqref{eqn:state-space-MMDc}.
Completion labels $D$ distinguish events that can produce the same
queue state after refill.

We say that the completion mechanism is \emph{consistent under
restriction} if
\begin{equation}\label{eq:completion-restriction-consistency}
 \delta_D(B)=
 \sum_{\substack{\emptyset\ne E\subseteq C\\E\cap B=D}}
       \delta_E(C),
 \qquad \emptyset\ne D\subseteq B\subseteq C\subseteq\mathcal I.
\end{equation}
The sum groups events on $C$ by their visible effect on $B$:
occupancy outside $B$ does not alter its marginal completion rates.
This condition concerns labelled completions, not projection of the
entire queue dynamics.

\begin{theorem}[Poisson shock representation]
\label{thm:queue-state-Markovianity}
For the Markovian completion mechanism specified above, consistency
under restriction in \eqref{eq:completion-restriction-consistency} is
equivalent to the existence of a representation by mutually independent
Poisson shock processes as in Section~\ref{subsec:shock-construction}.
The shock rates are uniquely determined by
\begin{equation}\label{eq:characterization-full-busy-rates}
 \mu_A=\delta_A(\mathcal I),
 \qquad \emptyset\ne A\subseteq\mathcal I,
\end{equation}
and all completion rates are then
\begin{equation}\label{eq:characterization-marginal-rates}
 \delta_D(B)=\sum_{\chg{A\in\mathcal P:\,A\cap B=D}}\mu_A,
 \qquad \emptyset\ne D\subseteq B\subseteq\mathcal I.
\end{equation}
The resulting queue-state process has exactly the generator of
Theorem~\ref{thm:MMDc-CTMC}.

In the shock representation, let $T_i(\tau)$ be the time from an
\chg{almost surely finite $(\mathcal F_t)$-stopping time $\tau$} to the next shock involving
server $i$. Conditional on the driving history $\mathcal F_\tau$,
for every nonempty $B\subseteq\mathcal I$ and $t_i\ge0$,
\begin{equation}\label{eq:characterization-residual-law}
 \mathbb P\bigl(T_i(\tau)>t_i,\ i\in B\mid\mathcal F_\tau\bigr)
 =\exp\left\{-\sum_{\chg{A\in\mathcal P:\,A\cap B\ne\emptyset}}
       \mu_A\max_{i\in A\cap B}t_i\right\}.
\end{equation}
For the servers busy immediately after $\tau$, this is the joint law
of their residual service-completion times. In particular, the mechanism has
Marshall--Olkin dependent residual completions and positive exponential
marginals with rates $\nu_i=\sum_{\chg{A\in\mathcal P:\,i\in A}}\mu_A$.
\end{theorem}
\begin{proof}
\emph{Recovery of the shock rates.}
Suppose first that \eqref{eq:completion-restriction-consistency} holds.
Set $C=\mathcal I$ and define $\mu_A$ by
\eqref{eq:characterization-full-busy-rates}. Each $\mu_A$ is
nonnegative and finite, and consistency gives
\eqref{eq:characterization-marginal-rates}. Taking $B=D=\{i\}$
shows that $\nu_i=\delta_{\{i\}}(\{i\})>0$.

\emph{Construction and equality in law.}
Construct independent Poisson processes $N_A$ with rates $\mu_A$,
independently of arrivals and assignment marks, and apply the event
rules of Section~\ref{subsec:shock-construction}. When the busy set is $B$, a type-$A$ shock
completes exactly the services in $A\cap B$. Summing over all
$A$ with $A\cap B=D$ gives the prescribed rate $\delta_D(B)$.
\begin{chgblock}
Shocks missing $B$ leave the queue unchanged and are suppressed when the
observed marked completion process is recorded. Arrivals and all subsequent
assignments also have their prescribed laws. Thus the original mechanism
and this construction have the same rates for every visible labelled
completion and the same assignment kernels.
\end{chgblock}
These bounded labelled event rates and kernels determine the joint law of
successive event times, completion labels, and assignments: condition on the
first event and iterate the same construction after each event. Nonexplosion makes this recursion
valid on every finite time interval. Hence the complete observed marked
processes agree in law, and their queue-state generators also agree.
This does not identify any pre-existing, unobserved physical drivers
as independent Poisson processes; it constructs such a representation.

\emph{Converse and uniqueness.}
Conversely, in any independent Poisson shock representation,
\eqref{eq:characterization-marginal-rates} follows by grouping shocks
according to their intersection with $B$. For $D\subseteq B\subseteq C$,
with $D\ne\emptyset$, grouping a second time gives
\[
 \sum_{\substack{\emptyset\ne E\subseteq C\\E\cap B=D}}
       \delta_E(C)
 =\sum_{\substack{\emptyset\ne E\subseteq C\\E\cap B=D}}
       \sum_{\chg{A\in\mathcal P:\,A\cap C=E}}\mu_A
 =\sum_{\chg{A\in\mathcal P:\,A\cap B=D}}\mu_A
 =\delta_D(B).
\]
This proves consistency. With every server busy, a shock of type $A$
completes exactly the services in $A$, so its rate must be
$\delta_A(\mathcal I)$. This also proves uniqueness of the rates.

\emph{Residual completion law.}
After $\tau$, the independent Poisson drivers have fresh increments,
independent of $\mathcal F_\tau$. For a fixed $A$ intersecting $B$,
the survival event in \eqref{eq:characterization-residual-law} excludes
a type-$A$ shock for a duration
$\max_{i\in A\cap B}t_i$. Multiplying these no-shock probabilities
gives the displayed law. Nonpreemptive service ensures that the next
shock involving a currently busy server completes its present customer,
regardless of arrivals or completions elsewhere. The formula therefore
describes the residual service times of the active customers.
\end{proof}

The three-server marginal rates in Example~\ref{exa:stability-3}
illustrate \eqref{eq:characterization-marginal-rates}; for example,
$\delta_{\{1\}}(\{1,2\})=\mu_1+\mu_{13}$.
Lemma~\ref{lem:closure} gives the weak lack-of-memory property for
every marginal residual vector, and
Theorem~\ref{thm:MO-characterization} gives its distributional converse.
These vectors are observed at a common time, rather than at different
service initiations.

Brigo, Mai and Scherer~\cite{Brigo-Mai-Scherer:2016} characterize
the Marshall--Olkin law through Markovian survival subvectors.
Here the consistency assumption is stated directly for labelled
queue-completion rates, yielding a shock representation of the
observed marked process.

\begin{remark}[Role of the consistency condition]
\label{rem:characterization-consistency-needed}
\normalfont
Queue-state Markovianity alone is insufficient for the converse.
For example, consider two servers with separate single completions,
each at rate $1$ when both servers are busy, but with completion rate
$2$ when only one server is busy. Use the same Poisson arrivals and
assignment rules as above. The queue-state process is a continuous-time
Markov chain without recording service ages. Nevertheless,
\[
 \delta_{\{1\}}(\{1\})=2
 \ne 1
 =\delta_{\{1\}}(\{1,2\})+
       \delta_{\{1,2\}}(\{1,2\}),
\]
so the rates are not consistent under restriction. No fixed family of
occupancy-independent Poisson shock rates can generate this mechanism.
Theorem~\ref{thm:queue-state-Markovianity} characterizes the
Marshall--Olkin representation within the stated class of consistent
Markovian completion mechanisms.
\end{remark}

\begin{remark}[The consistency condition and the earlier $M/M_D/2$ model]
\label{rem:consistency-earlier-two-server}
\normalfont
The $M/M_D/2$ model of Thapa and Zhao~\cite{Thapa-Zhao:2026}
was formulated with Marshall--Olkin dependent service mechanisms,
but consistency under restriction was not stated explicitly as a
separate condition on general labelled completion rates.
For a model constructed from fixed Poisson shocks, the condition
\eqref{eq:completion-restriction-consistency} holds automatically.
It becomes essential when proceeding in the opposite direction:
Markovianity of the queue-state process alone does not ensure that
an arbitrary completion mechanism admits a Marshall--Olkin
representation. Theorem~\ref{thm:queue-state-Markovianity}
identifies the additional compatibility requirement, for two
servers as well as for general $c$.
The characterization clarifies the scope of the modelling
assumptions without changing the results derived for the
Marshall--Olkin $M/M_D/2$ queue.
\end{remark}

\begin{remark}[Extension to other queueing architectures]
\label{rem:MO-service-Markovianity-networks}
\normalfont
Other finite queueing networks admit the same construction when each
finite-rate driving event has a state-based update using independent
marks. A shock completes existing services before routing and new
service initiations. A converse for another architecture requires
its own consistency assumptions.
\end{remark}

\section{The stationary departure process and output performance}
\label{sec:departures}

\begin{chgblock}
The classical $M/M/c$ queue with identical independent servers has
stationary Poisson output by Burke's theorem~\cite{Burke:1956}.
Heterogeneity alone does not preclude this property: certain heterogeneous
exponential-server queues are quasi-reversible under suitable idle-server
assignment rules and therefore also have stationary Poisson output; see
Kelly~\cite{Kelly:1982QR}. Quasi-reversibility is sufficient, but not necessary, for Poisson output. Under more general
heterogeneous assignment rules, however, Poisson output is not automatic.
The shared-shock model considered here introduces a further structural
difference, since a single shock can generate simultaneous departures.
Retain $Y$, $X$, $Q$, and the assignment rule of
Section~\ref{sec:queue-model}, assume $0<\lambda<\Theta$, and let
$\boldsymbol{\pi}=(\pi_y:y\in S)$ denote the stationary probability row
vector. We write $\mathbb P_\pi$ and $\mathbb E_\pi$ for probability and
expectation under the stationary law. Marked-generator methods
\cite{Lucantoni:1991} yield batch, count, and interdeparture measures;
flow conservation determines the long-run count fluctuations.
\end{chgblock}

\subsection{Marked generator and matrix-analytic representation}
\label{subsec:departure-generator}

Let $N_k(t)$ count departure epochs during $(0,t]$ at which exactly
$k$ customers leave, and let $D(t)$ count departed customers. Thus
\begin{equation}\label{eq:dep-counts}
 D(t)=\sum_{k=1}^c kN_k(t),\qquad D(0)=N_k(0)=0.
\end{equation}
\begin{chgblock}
A shock of type $A$ occurring at time $s\in(0,t]$ has departure mark
$|A\cap B(Y(s^-))|$, the number of affected servers that are busy
immediately before the shock; the mark is zero if it affects only idle servers.
\end{chgblock}

For $1\le k\le c$, define the nonnegative matrices
\begin{equation}\label{eq:dep-matrices}
 (D_k)_{y,v}=q_{y,v}\,
       \mathbf1_{\{x(y)-x(v)=k\}},\qquad
 D_+=\sum_{k=1}^cD_k,\qquad D_0=Q-D_+.
\end{equation}
All $D_k$ have zero diagonals. The off-diagonal entries of $D_0$ are
arrivals, and its diagonal is $q_{y,y}=-(\lambda+\alpha(B(y)))$;
shocks affecting only idle servers remain invisible. \begin{chgblock}
Equation~\eqref{eq:dep-matrices} classifies the transition rates in
Theorem~\ref{thm:MMDc-CTMC} by the number of departing customers.
The destination states and transition rates already account for the
immediate refilling of available servers by waiting customers.
\end{chgblock}

The joint process $(Y,D)$ is batch Markov additive, with a
\emph{countable-state} BMAP representation. The finite QBD level size
does not make the background finite, nor does this representation
imply conditionally Poisson departures given the background trajectory.
Define

\begin{equation}\label{eq:dep-marked-generator}
 \mathcal Q(z)=D_0+\sum_{k=1}^c z^kD_k,\qquad 0\le z\le1.
\end{equation}
\begin{chgblock}
The matrix $\mathcal Q(z)$ has the same block structure as $Q$. Denote its
repeating blocks by $A_{-1}(z),A_0(z),A_1(z)$ and its boundary blocks by
$B_{-1}(z),B_0(z),B_1(z)$, corresponding to the blocks introduced in
Section~\ref{sec:queue-model}. Upward-level transitions correspond to
arrivals and carry departure mark zero; hence $A_1(z)=A_1$.
Writing $\Sigma=\sum_{k=1}^cR_k$, the repeating blocks are
\end{chgblock}
\begin{align}
 A_1(z)&=A_1,&
 (A_0(z))_{ij}&=
 \begin{cases}
 -(\lambda+\Sigma),&j=i,\\
 \lambda,&j=i+1,\\
 z^{i-j}R_{i-j},&j<i,\\
 0,&j>i+1,
 \end{cases}
 \label{eq:dep-marked-repeating}\\
 (A_{-1}(z))_{ij}
   &=\begin{cases}
      z^{c+i-j}R_{c+i-j},&j\ge i,\\
      0,&j<i.
     \end{cases}
 \label{eq:dep-marked-down}
\end{align}
The boundary blocks satisfy $B_1(z)=B_1$,
\[
 (B_0(z))_{s,t}=z^{|s|-|t|}(B_0)_{s,t}\quad(t\subsetneq s),
 \qquad
 (B_{-1}(z))_{j,s}
     =z^{c+j-1-|s|}(B_{-1})_{j,s};
\]
all diagonal and arrival entries of $B_0$ are unchanged. Thus the
explicit blocks already derived supply the marked blocks directly.

\begin{proposition}[Departure-count transform]
\label{prop:dep-transform}
For an initial queue-state distribution $a$,
\begin{equation}\label{eq:dep-pgf}
 \mathbb E_a[z^{D(t)}]=a e^{\mathcal Q(z)t}\mathbf1,\qquad
 \mathbb P_a(D(t)=m)=[z^m]\,a e^{\mathcal Q(z)t}\mathbf1.
\end{equation}
More precisely,
\[
 (e^{\mathcal Q(z)t})_{y,v}
   =\mathbb E_y[z^{D(t)}\mathbf1_{\{Y(t)=v\}}].
\]
\end{proposition}
\begin{proof}
For the joint process, an arrival moves $Y$ without changing $D$,
whereas a transition represented by $D_k$ increases $D$ by $k$.
Weighting the latter transition by $z^k$ gives the backward equation
with operator $\mathcal Q(z)$ and initial matrix $I$. All rates are
bounded, so uniformization yields its unique solution
$e^{\mathcal Q(z)t}$. Summation over the final state and coefficient
extraction give \eqref{eq:dep-pgf}. \end{proof}

In the homogeneous model, no symmetry of the assignment probabilities
is needed for the customer-count reduction. On $\{0,1,\ldots\}$,
\begin{equation}\label{eq:dep-homogeneous-matrices}
 (D_k)_{n,n-k}=R_{\min(n,c),k}\quad(n\ge k),\qquad
 (D_0)_{n,n+1}=\lambda,\qquad
 (D_0)_{n,n}=-(\lambda+\Sigma_{\min(n,c)}),
\end{equation}
with all other entries zero and $\Sigma_0=0$. These matrices can replace
the labelled-state matrices in every count formula below.

\subsection{Throughput, batch sizes, and departure-epoch probabilities}
\label{subsec:departure-batches}

Put
\begin{equation}\label{eq:dep-stationary-rates}
 b_k=\chg{\boldsymbol{\pi}}D_k\mathbf1,\qquad
 \beta=\sum_{k=1}^c b_k,\qquad
 P_{\rm full}=\mathbb P_\pi(X\ge c).
\end{equation}
\begin{chgblock}
Here $b_k$ is the stationary rate of size-$k$ departure epochs, and
$\beta$ is the total stationary departure-epoch rate (the expected number
of departure epochs per unit time). Let $\mathbb P^{\rm dep}$ and
$\mathbb E^{\rm dep}$ denote probability and expectation under stationary
departure-epoch sampling, with a simultaneous batch counted as one departure epoch.
\end{chgblock}

\begin{proposition}[Stationary batch measures]
\label{prop:dep-batch-measures}
The customer throughput and mean counts are
\begin{equation}\label{eq:dep-throughput}
 \sum_{k=1}^c k b_k=\lambda,\qquad
 \mathbb E_\pi[N_k(t)]=b_k t,\qquad
 \mathbb E_\pi[D(t)]=\lambda t.
\end{equation}
For a typical stationary departure epoch, let $K$ be its batch size.
Then
\begin{equation}\label{eq:dep-batch-law}
 \mathbb P^{\rm dep}(K=k)=\frac{b_k}{\beta},\qquad
 \mathbb E^{\rm dep}[K]=\frac{\lambda}{\beta}.
\end{equation}
The fraction of departed customers belonging to a size-$k$ batch is
$k b_k/\lambda$. The queue-state distribution immediately after a
typical departure epoch is
\begin{equation}\label{eq:dep-palm-state}
 \omega=\frac{\chg{\boldsymbol{\pi}}D_+}{\beta}.
\end{equation}
\end{proposition}
\begin{proof}
\begin{chgblock}
Under stationarity, when the queue state immediately before time $u$ is $y$,
the instantaneous rate of size-$k$ departure epochs is $(D_k\mathbf1)_y$.
Hence
\[
 \mathbb E_\pi[N_k(t)]
 =\int_0^t\mathbb E_\pi[(D_k\mathbf1)_{Y(u^-)}]\,du
 =t\,\boldsymbol{\pi}D_k\mathbf1
 =b_kt.
\]
\end{chgblock}
The conservation identity
\[
 D(t)=N_\lambda(t)+X(0)-X(t)
\]
and stationarity give $\mathbb E_\pi D(t)=\lambda t$.
The geometric stationary tail ensures the required finite mean.
The stationary rate of size-$k$ transitions ending in state $v$ is
$(\chg{\boldsymbol{\pi}}D_k)_v$. Normalizing these rates by the total epoch rate
$\beta$ proves \eqref{eq:dep-batch-law} and
\eqref{eq:dep-palm-state}. Weighting a size-$k$ epoch by its $k$
departed customers gives the customer fraction.
\end{proof}

\begin{chgblock}
Flow conservation fixes the stationary customer throughput at $\lambda$, but
does not determine how departures are grouped into batches. Consequently,
systems with the same arrival rate can have different departure-epoch rates
and batch-size distributions.
\end{chgblock}

These measures require only finite sums after using the scalar tail:
\begin{equation}\label{eq:dep-finite-rates}
 b_k=\sum_{s\subsetneq\mathcal I}\pi_s
              \sum_{\chg{A\in\mathcal P:\,|A\cap s|=k}}\mu_A
       +P_{\rm full}R_k,\qquad
 P_{\rm full}=\frac{r}{1-r}
                  \sum_{|s|=c-1}\pi_s.
\end{equation}
For homogeneous servers this reduces to
\[
 b_k=\sum_{m=k}^{c-1}\pi_mR_{m,k}
          +P_{\rm full}R_{c,k},\qquad
 P_{\rm full}=\frac{\pi_{c-1}r}{1-r}.
\]
In particular, if some $\mu_A>0$ with $|A|\ge2$, then
$b_{|A|}\ge P_{\rm full}\mu_A>0$. Consequently, in stationarity,
\[
 \mathbb P_\pi(D(h)\ge2)
       =h\sum_{k=2}^c b_k+O(h^2).
\]
An ordinary Poisson process has probability $O(h^2)$ of two or more
points in an interval of length $h$. Thus positive shared-shock rates
rule out an ordinary Poisson departure process.

\subsection{A finite phase-type law for interdeparture times}
\label{subsec:departure-intervals}

\begin{chgblock}
Under $\mathbb P^{\rm dep}$, let $\tau_d$ be the positive interval from
the sampled stationary departure epoch to the next, and let $K_{\rm next}$
be the next batch size; a simultaneous batch counts as one epoch.
\end{chgblock} Before the next departure, arrivals only add
busy servers; once all are busy, later arrivals cannot affect that
departure's time or batch size. Hence one interval admits a finite
representation despite the countably infinite output background.

Index a matrix $\mathsf T$ by the $2^c$ subsets
$\mathcal B=\{B:B\subseteq\mathcal I\}$, in increasing cardinality.
Its nonzero entries are
\begin{equation}\label{eq:dep-finite-subgenerator}
 \begin{aligned}
 \mathsf T_{B,B}&=-(\lambda+\alpha(B)),
                  &&B\subsetneq\mathcal I,\\
 \mathsf T_{B,B\cup\{i\}}&=\lambda p_i(\mathcal I\setminus B),
                  &&i\notin B,\\
 \mathsf T_{\mathcal I,\mathcal I}&=-\Sigma.
 \end{aligned}
\end{equation}
The fully busy diagonal excludes arrivals, since they leave this
finite phase unchanged. For each $k$, define the absorption column
\[
 (\mathbf r_k)_B=\sum_{\chg{A\in\mathcal P:\,|A\cap B|=k}}\mu_A.
\]
Then $-\mathsf T\mathbf1=\sum_{k=1}^c\mathbf r_k$. Aggregate the
post-departure distribution in \eqref{eq:dep-palm-state} to
\begin{equation}\label{eq:dep-finite-initial}
 \zeta_B=\sum_{y:\,B(y)=B}\omega_y,\qquad B\in\mathcal B.
\end{equation}

\begin{theorem}[Interdeparture times and the next batch]
\label{thm:dep-interdeparture}
The stationary interdeparture time is phase-type with representation
$(\zeta,\mathsf T)$. In particular, for $t\ge0$, $s\ge0$, and
integers $m\ge1$,
\begin{align}
 \mathbb P^{\rm dep}(\tau_d>t)
    &=\zeta e^{\mathsf Tt}\mathbf1,
 \label{eq:dep-interval-survival}\\
 \mathbb P^{\rm dep}(\tau_d\in dt,K_{\rm next}=k)
    &=\zeta e^{\mathsf Tt}\mathbf r_k\,dt,
 \label{eq:dep-interval-joint}\\
 \mathbb E^{\rm dep}[e^{-s\tau_d}
                   \mathbf1_{\{K_{\rm next}=k\}}]
    &=\zeta(sI-\mathsf T)^{-1}\mathbf r_k,
 \label{eq:dep-interval-transform}\\
 \mathbb E^{\rm dep}[\tau_d^m]
    &=m!\,\zeta(-\mathsf T)^{-m}\mathbf1,\qquad
 \mathbb E^{\rm dep}[\tau_d]=\frac1\beta.
 \label{eq:dep-interval-moments}
\end{align}
\end{theorem}
\begin{proof}
Kill the process at its next departure epoch. From a proper busy set,
an arrival adds one server according to $p_i$; any shock meeting the
busy set causes absorption. From the fully busy set, the first
service shock causes absorption at total rate $\Sigma$, independently
of intervening arrivals. These are precisely the transient and
absorption rates in \eqref{eq:dep-finite-subgenerator}.
The initial phase has distribution $\zeta$ by stationary event-rate
sampling. This proves the survival and marked density formulas.
Integration gives the transform and moments.

Every nonempty busy set has strictly positive departure rate, since
$\nu_i>0$ for each server. The empty phase is left at rate $\lambda>0$.
Hence this finite transient chain is absorbed almost surely with
finite moments, and $-\mathsf T$ is invertible.

To verify the mean explicitly, the corresponding killed generator
on the full state space is $D_0$. Its mean lifetime is
$(-D_0)^{-1}\mathbf1$, and aggregation gives
$\omega(-D_0)^{-1}\mathbf1=\zeta(-\mathsf T)^{-1}\mathbf1$.
From $\chg{\boldsymbol{\pi}}Q=0$ and \eqref{eq:dep-palm-state},
$\beta\omega=\chg{\boldsymbol{\pi}}D_+=-\chg{\boldsymbol{\pi}}D_0$. Therefore
\[
 \mathbb E^{\rm dep}[\tau_d]
 =\omega(-D_0)^{-1}\mathbf1=\frac{\chg{\boldsymbol{\pi}}\mathbf1}{\beta}
 =\frac1\beta.
\]
The killed inverse is well defined: after the first arrival, the
departure hazard is bounded below by $\min_i\nu_i>0$.
\end{proof}

\begin{proposition}[Marked Markov-renewal description]
\label{prop:dep-marked-renewal}
\begin{chgblock}
Under stationary departure-epoch sampling, shift the sampled departure
epoch to time $0$ and let $\sigma_0=0<\sigma_1<\sigma_2<\cdots$ denote
subsequent departure epochs. Put $\tau_n=\sigma_n-\sigma_{n-1}$,
$J_n=Y(\sigma_n)$, and let $K_n$ be the batch size at $\sigma_n$.
Then $(J_n,\sigma_n)$ is a Markov-renewal process with batch marks $K_n$
and transition kernel
\begin{equation}\label{eq:dep-renewal-kernel}
 \mathbb P^{\rm dep}\{\tau_{n+1}\in dt,
      K_{n+1}=k,J_{n+1}=v\mid J_n=y\}
      =(e^{D_0t}D_k)_{y,v}\,dt,
 \qquad t>0.
\end{equation}
Consequently, for every $m\ge1$, $t_i>0$, and
$k_i\in\{1,\ldots,c\}$,
\begin{equation}\label{eq:dep-joint-intervals-batches}
\begin{aligned}
 &\mathbb P^{\rm dep}\{\tau_1\in dt_1,K_1=k_1,\ldots,
                       \tau_m\in dt_m,K_m=k_m\}\\
 &\qquad=\omega e^{D_0t_1}D_{k_1}\cdots
            e^{D_0t_m}D_{k_m}\mathbf1\,dt_1\cdots dt_m.
\end{aligned}
\end{equation}
\end{chgblock}
\end{proposition}
\begin{proof}
\begin{chgblock}
The killed semigroup $e^{D_0t}$ describes evolution without a departure,
and multiplication by $D_k$ records the next size-$k$ departure and its
destination state, including immediate refill. The strong Markov property
at $\sigma_n$ shows that the conditional kernel depends on the past only through
$J_n$. Iterating this kernel from the initial distribution
$\omega$ in \eqref{eq:dep-palm-state} gives
\eqref{eq:dep-joint-intervals-batches}.
\end{chgblock}
\end{proof}

\begin{example}[Independent heterogeneous servers: Poisson versus non-Poisson output]
\label{exa:hetero-independent-output}
\begin{chgblock}
Consider $c=2$, $\lambda=1$, $\mu_1=2$, $\mu_2=1$, and no shared
shock ($\mu_{12}=0$). First take the server-$1$-first rule $p=1$.
Corollary~\ref{cor:hetero-MAM-independent-two} gives
\[
 \pi_0=\frac{10}{19},\qquad
 \pi_{(1)}=\frac4{19},\qquad
 \pi_{(2)}=\frac2{19},\qquad
 \pi_2=\frac2{19},
\]
with $\pi_n=(2/19)(1/3)^{n-2}$ for $n\ge2$. In the busy-set order
$\emptyset,\{1\},\{2\},\{1,2\}$, Theorem~\ref{thm:dep-interdeparture}
gives
\[
 \zeta=\left(\frac{10}{19},\frac2{19},\frac4{19},\frac3{19}\right),
 \qquad
 \mathsf T=
 \begin{pmatrix}
 -1&1&0&0\\
 0&-3&0&1\\
 0&0&-2&1\\
 0&0&0&-3
 \end{pmatrix}.
\]
Hence
\[
 \mathbb P^{\rm dep}(\tau_d>t)
 =\frac{35e^{-t}+16e^{-2t}-(13+6t)e^{-3t}}{38},
\]
which is not exponential; for example, its density at $0+$ is $17/19$
rather than $1$. Thus independent heterogeneous exponential servers need
not have Poisson output under an arbitrary identity-dependent assignment rule.

If instead arrivals choose uniformly among idle servers, $p=1/2$, Kelly's
quasi-reversibility result~\cite{Kelly:1982QR} gives stationary Poisson
output of rate $\lambda=1$. Consistently, the present phase-type calculation
reduces to $\mathbb P^{\rm dep}(\tau_d>t)=e^{-t}$.
\end{chgblock}
\end{example}

For proper $B$, only states with at most $|B|+c$ customers can depart
into $B$, so $\zeta_B$ is a finite sum. A size-$k$ departure leaves
all servers busy exactly when the pre-departure count is at least $c+k$;
therefore
\begin{equation}\label{eq:dep-full-post-epoch}
 \zeta_{\mathcal I}
   =\frac{P_{\rm full}}{\beta}\sum_{k=1}^cR_kr^k.
\end{equation}

For homogeneous servers, $\mathsf T$ reduces to the $c+1$ phases
$m=0,\ldots,c$, with
\[
 \mathsf T_{m,m}=-(\lambda+\Sigma_m),\quad
 \mathsf T_{m,m+1}=\lambda\quad(m<c),\qquad
 \mathsf T_{c,c}=-\Sigma_c,\qquad
 (\mathbf r_k)_m=R_{m,k},
\]
where $R_{m,k}=0$ for $k>m$. This reduction again holds for every
assignment rule of Section~\ref{subsec:shock-construction}.

The squared coefficient of variation of the epoch interval is
\begin{equation}\label{eq:dep-gap-scv}
 \operatorname{SCV}(\tau_d)
   =2\beta^2\zeta(-\mathsf T)^{-2}\mathbf1-1.
\end{equation}
At an arbitrary stationary time, use instead
$\xi_B=\pi_B$ for proper $B$ and $\xi_{\mathcal I}=P_{\rm full}$.
This generally differs from departure-epoch sampling.
\begin{chgblock}
This finite phase-type representation describes the marginal law of one
stationary interdeparture interval, together with the size of the departure
that terminates it. It does not imply that successive interdeparture intervals
are independent, nor that the full departure process has a BMAP representation
with a finite background state space. Indeed, after a departure, the busy-server
set alone need not determine the future evolution: the number of customers still
waiting may also matter, and this information is not retained in the finite
phase description.
\end{chgblock}

\subsection{Finite-time variability and long-run fluctuations}
\label{subsec:departure-variability}

Define the bounded matrices
\begin{equation}\label{eq:dep-moment-matrices}
 M_1=\sum_{k=1}^c kD_k,\qquad
 M_2=\sum_{k=1}^c k^2D_k,\qquad
 v_0=\chg{\boldsymbol{\pi}}M_2\mathbf1=\sum_{k=1}^c k^2b_k.
\end{equation}

\begin{theorem}[Departure-count variance and limiting fluctuations]
\label{thm:dep-variance}
In stationarity,
\begin{equation}\label{eq:dep-finite-variance}
 \operatorname{Var}_\pi D(t)
 =v_0t+2\int_0^t(t-u)
       \left[\chg{\boldsymbol{\pi}} M_1e^{Qu}M_1\mathbf1-\lambda^2\right]\,du.
\end{equation}
Moreover,
\begin{equation}\label{eq:dep-long-variance}
 \lim_{t\to\infty}\frac{\operatorname{Var}_\pi D(t)}{t}
      =\lambda,\qquad
 \frac{D(t)-\lambda t}{\sqrt t}
      \ \Longrightarrow\ \mathcal N(0,\lambda).
\end{equation}
For the index of dispersion
$I_D(t)=\operatorname{Var}_\pi D(t)/\mathbb E_\pi D(t)$,
\begin{equation}\label{eq:dep-dispersion}
 \lim_{t\downarrow0}I_D(t)
    =\frac{\sum_{k=1}^c k^2b_k}{\lambda}
    =1+\frac{\sum_{k=2}^c k(k-1)b_k}{\lambda},
 \qquad
 \lim_{t\to\infty}I_D(t)=1.
\end{equation}
\end{theorem}
\begin{proof}
Differentiate the transform in \eqref{eq:dep-pgf} at $z=1$.
The first two derivatives of $\mathcal Q(z)$ there are $M_1$ and
$M_2-M_1$. Using $\chg{\boldsymbol{\pi}}e^{Qu}=\chg{\boldsymbol{\pi}}$ and $e^{Qu}\mathbf1=\mathbf1$
gives
\[
 \mathbb E_\pi[D(t)(D(t)-1)]
 =t\chg{\boldsymbol{\pi}}(M_2-M_1)\mathbf1
     +2\int_0^t(t-u)\chg{\boldsymbol{\pi}} M_1e^{Qu}M_1\mathbf1\,du.
\]
Adding $\mathbb E_\pi D(t)=\lambda t$ and subtracting $\lambda^2t^2$
proves \eqref{eq:dep-finite-variance}. Differentiation is legitimate
because the event rate is bounded and each departure batch has size
at most $c$.

For the long-run results, write
\[
 D(t)-\lambda t
   =[N_\lambda(t)-\lambda t]+[X(0)-X(t)].
\]
The scalar geometric stationary tail gives $\mathbb E_\pi X^2<\infty$.
Consequently, with $Z_t=X(0)-X(t)$,
$\mathbb E_\pi Z_t=0$ and
$\mathbb E_\pi Z_t^2\le4\mathbb E_\pi X^2$ uniformly in $t$.
Cauchy--Schwarz therefore yields
\[
 \left|\operatorname{Var}_\pi D(t)-\lambda t\right|
 \le4\mathbb E_\pi X^2
        +4\sqrt{\lambda t\,\mathbb E_\pi X^2}=o(t).
\]
Also $Z_t/\sqrt t\to0$ in $L^2$. The Poisson central limit theorem
and Slutsky's theorem prove the second assertion in
\eqref{eq:dep-long-variance}, without any independence assumption
between $Z_t$ and the arrival count.
Finally, the integral in \eqref{eq:dep-finite-variance} is $O(t^2)$
near zero. This gives both limits in \eqref{eq:dep-dispersion}.
\end{proof}

Shared shocks increase short-interval customer-count dispersion.
The long-run variance rate follows from conservation and a finite
stationary second moment; it implies neither Poisson output nor
independent increments. Dependence remains in the finite-time integral.

\paragraph{Computation of finite-time counts.}
Let $\kappa=\lambda+\sum_A\mu_A$. The marked kernel
$P(z)=I+\mathcal Q(z)/\kappa$ is nonnegative and substochastic for
$0\le z\le1$, and
\begin{equation}\label{eq:dep-uniformization}
 e^{\mathcal Q(z)t}
   =e^{-\kappa t}\sum_{j=0}^\infty
      \frac{(\kappa t)^j}{j!}P(z)^j.
\end{equation}
Thus count probabilities can be computed by polynomial matrix
recursions using the same QBD blocks. If the initial stationary sum
is restricted, without renormalizing, to $X(0)\le L$ with $L\ge c-1$,
and the event sum to $j\le M$, only states with at most $L+M$
customers are required. The omitted probability mass, and hence the
absolute error in each count probability or in the PGF on $[0,1]$,
is at most
\[
 P_{\rm full}r^{L-c+1}
       +\mathbb P\{\operatorname{Poisson}(\kappa t)>M\}.
\]
This gives a controlled finite computation of the countable-state
representation.

\subsection{Two-server and classical reductions}
\label{subsec:departure-two}

For $c=2$, write $u=\mu_{12}$ and $p=p_1(\{1,2\})$. Then
\begin{equation}\label{eq:dep-two-batch-metrics}
 b_2=uP_{\rm full},\qquad
 b_1=\lambda-2uP_{\rm full},\qquad
 \beta=\lambda-uP_{\rm full}.
\end{equation}
The double-batch probability at a departure epoch is
$uP_{\rm full}/\beta$; the fraction of customers in double batches
is $2uP_{\rm full}/\lambda$, and
$I_D(0+)=1+2uP_{\rm full}/\lambda$.
The finite matrix \eqref{eq:dep-finite-subgenerator} has four busy-set
phases. Example~\ref{exa:waiting-two-performance} combines its numerical
output measures with the corresponding delay measures.

For identical independent servers, Burke's theorem gives
$b_1=\beta=\lambda$, independent exponential interdeparture times
of rate $\lambda$, and $I_D(t)=1$ for every $t>0$.

\section{Virtual waiting time and sojourn time}
\label{sec:waiting-sojourn}

Under the FCFS and assignment rules of Section~\ref{subsec:shock-construction}
and $0<\lambda<\Theta$, let $W_q$ and $T$ denote a stationary arrival's
waiting and sojourn times. The virtual waiting time $V(t)$ is the time
to service for a hypothetical customer joining at $t$ under the same
future shocks; it is distinct from total unfinished service.
By PASTA~\cite{Wolff:1982}, the arrival sees the stationary law. Future shocks and
marks retain their laws, and later arrivals cannot affect its wait,
so stationary $V(t)$ and $W_q$ have the same distribution.

\subsection{Conditional waiting times}
\label{subsec:conditional-waiting}

An arrival seeing a proper busy subset enters service immediately.
If it sees $n\ge c$ customers, its waiting-line position is
\begin{equation}\label{eq:wait-rank}
 j=n-c+1.
\end{equation}
All servers remain busy until its service begins, and each type-$A$
shock supplies $|A|$ openings in FCFS order.
\begin{chgblock}
Let $\tau$ denote the tagged customer's arrival time and define
\[
 \widehat N_A(u)=N_A(\tau+u)-N_A(\tau),\qquad u\ge0.
\]
The Poisson driving processes have stationary independent increments; since $\tau$ is a random arrival time, the strong Markov property implies that, conditional on $\mathcal F_\tau$, the processes $\{\widehat N_A:A\in\mathcal P\}$ are again mutually independent Poisson processes with respective rates $\mu_A$. The potential completion count is
\begin{equation}\label{eq:wait-potential-count}
 C(u)=\sum_{A\in\mathcal P}|A|\widehat N_A(u).
\end{equation}
By superposition, $C$ is a compound Poisson process with jump rates $R_k$.
Equivalently, in the standard representation
\[
 C(u)=\sum_{n=1}^{N(u)}J_n,
\]
where $N$ is Poisson of rate $\Sigma=\sum_{k=1}^cR_k$ and the i.i.d.
jump sizes satisfy $\mathbb P(J_n=k)=R_k/\Sigma$, $1\le k\le c$.
\end{chgblock}
Therefore the conditional wait is
\begin{equation}\label{eq:wait-first-passage}
 W_j=\inf\{u\ge0:C(u)\ge j\}.
\end{equation}
Excess openings cause immediate service entry.
\begin{chgblock}
Right before the tagged customer begins service, the tagged customer is still
waiting; hence, by the work-conserving rule, all $c$ servers are busy.
Therefore every type-$A$ shock up to and including the shock that initiates
the tagged customer's service produces exactly $|A|$ service completions.
After the tagged customer enters service, the waiting line may become empty,
so subsequent type-$A$ shocks may affect idle servers and produce fewer than
$|A|$ actual departures. Thus the continuation of $C$ need not coincide with
the actual departure-count process.
\end{chgblock}

\chg{On the transient states $1,\ldots,j$, where state $\ell$ represents $\ell$ additional service openings required before the tagged customer enters service,} define the subgenerator
\begin{equation}\label{eq:wait-conditional-generator}
 (\mathsf U_j)_{\ell m}=
 \begin{cases}
 -\Sigma,&m=\ell,\\
 R_{\ell-m},&1\le\ell-m\le c,\\
 0,&\text{otherwise},
 \end{cases}
 \qquad 1\le\ell,m\le j.
\end{equation}
Its absorption-rate vector is
\begin{equation}\label{eq:wait-conditional-exit}
 \boldsymbol t_j=-\mathsf U_j\mathbf1,\qquad
 (\boldsymbol t_j)_\ell=\sum_{k=\ell}^cR_k,
\end{equation}
where empty sums are zero. \chg{Let $e_j$ denote the $j$th standard basis column vector in $\mathbb R^j$, with a $1$ in position $j$ and zeros elsewhere.} Then $W_j\sim\mathrm{PH}(e_j^{\top},\mathsf U_j)$ and
$\mathbb P(W_j>t)=e_j^{\top}e^{\mathsf U_jt}\mathbf1$.
The standard phase-type formulas~\cite{Neuts:1981} give
\begin{align}
 \phi_j(s):=\mathbb E[e^{-sW_j}]
   &=e_j^{\top}(sI_j-\mathsf U_j)^{-1}\boldsymbol t_j,
 \label{eq:wait-conditional-transform}\\
 \mathbb E[W_j^m]&=m!\,e_j^{\top}(-\mathsf U_j)^{-m}\mathbf1,
 \qquad m\ge1.
 \label{eq:wait-conditional-moments}
\end{align}

Equivalently, first-step analysis gives
\begin{equation}\label{eq:wait-recursions}
 (s+\Sigma)\phi_j(s)=\sum_{k=1}^cR_k\phi_{j-k}(s),\qquad
 \Sigma w_j=1+\sum_{k=1}^cR_kw_{j-k},\qquad j\ge1,
\end{equation}
where $\phi_j(s)=1$ and $w_j=0$ for $j\le0$, and $w_j=\mathbb E W_j$.
The $j=n-c+1$ waiting phases and their laws do not depend on assignment.

\subsection{An exact stationary waiting-time distribution}
\label{subsec:stationary-waiting}

Write $P_{\rm full}=\mathbb P_\pi(X\ge c)$, as in
\eqref{eq:dep-stationary-rates}, and let $r\in(0,1)$ be the root
in \eqref{eq:hetero-MAM-root}. The geometric tail in
Theorem~\ref{thm:hetero-MAM-stationary} gives
\begin{equation}\label{eq:wait-geometric-rank}
 \mathbb P(W_q>0)=P_{\rm full},\qquad
 \mathbb P(J=j\mid W_q>0)=(1-r)r^{j-1},\qquad j\ge1,
\end{equation}
where $J$ is the initial waiting-line position.
\chg{Thus the infinite geometric mixture of the conditional phase-type
waiting-time laws collapses to an exponential distribution for $W_q$
conditional on $W_q>0$.}

\begin{theorem}[Stationary waiting time]
\label{thm:stationary-waiting}
Define
\begin{equation}\label{eq:wait-rate}
 \vartheta=\sum_{k=1}^cR_k(1-r^k)
          =\Sigma-\sum_{k=1}^cR_kr^k
          =\frac{\lambda(1-r)}{r}>0.
\end{equation}
Then
\begin{align}
 \mathbb P(W_q=0)&=1-P_{\rm full},&
 \mathbb P(W_q>t)&=P_{\rm full}e^{-\vartheta t},\quad t\ge0,
 \label{eq:wait-stationary-law}\\
 \mathbb E[e^{-sW_q}]&=1-P_{\rm full}
             +P_{\rm full}\frac{\vartheta}{s+\vartheta},
 &&s\ge0.
 \label{eq:wait-stationary-transform}
\end{align}
Thus $W_q$ conditional on $W_q>0$ is exponential with rate
$\vartheta$, for both homogeneous and heterogeneous servers.
\end{theorem}
\begin{proof}
Future shock increments are independent of the state seen at arrival.
Conditional on positive waiting, $J$ therefore has the geometric law
in \eqref{eq:wait-geometric-rank} independently of $C$. For $t\ge0$,
\[
 \mathbb P(W_q>t\mid W_q>0)
   =\mathbb P(J>C(t)\mid W_q>0)
   =\mathbb E[r^{C(t)}]
   =\exp\!\left(t\sum_{k=1}^cR_k(r^k-1)\right).
\]
This proves the survival formula and hence the transform. Multiplying
\eqref{eq:hetero-MAM-root} by $(1-r)/r$ proves the last identity
in \eqref{eq:wait-rate}.

For the matrix-analytic interpretation, let $\mathsf U$ be the
countably infinite subgenerator whose entries are given by
\eqref{eq:wait-conditional-generator}, and put
$\boldsymbol a=((1-r)r^{j-1})_{j\ge1}$. Directly,
\begin{equation}\label{eq:wait-quasistationary}
 (\boldsymbol a\mathsf U)_j
  =-\Sigma a_j+\sum_{k=1}^cR_ka_{j+k}
  =-\vartheta a_j,
 \qquad
 \boldsymbol a e^{\mathsf Ut}=e^{-\vartheta t}\boldsymbol a.
\end{equation}
The bounded transition rates justify the semigroup identity.
Thus the geometric distribution of the remaining position is
unchanged conditional on continued waiting, and its survival
probability decreases at the constant rate $\vartheta$.
\end{proof}

The moments and variance are
\begin{equation}\label{eq:wait-performance}
 \mathbb E[W_q^m]=\frac{m!P_{\rm full}}{\vartheta^m},\quad m\ge1,
 \qquad \operatorname{Var}(W_q)
 =\frac{P_{\rm full}(2-P_{\rm full})}{\vartheta^2}.
\end{equation}
Thus $\vartheta$ is an exact decay rate; the moment generating function
is finite precisely below it. For $L_q=(X-c)^+$, the geometric tail gives
\begin{equation}\label{eq:wait-Little}
 \mathbb E_\pi L_q=P_{\rm full}\frac{r}{1-r}
             =\lambda\mathbb E W_q,
\end{equation}
as required by Little's law~\cite{Little:1961}. Assignment affects the
delay probability $P_{\rm full}$ through the boundary, while the
conditional positive-wait law depends only on $\lambda$ and $R_k$.

\subsection{Service assignment and conditional sojourn times}
\label{subsec:tagged-assignment}

Let $I_*$ be the selected server and $S$ the service duration, so
$T=W_q+S$. At initiation on server $i$, fresh shock increments give
$S\sim\mathrm{Exp}(\nu_i)$ independently of the past, conditional on
$I_*=i$.
Without conditioning on the server, $W_q$ and $S$ need not be independent.

To account for position within a simultaneous refill, let
$a_{i,\ell}(A)$ be the probability that the $\ell$th customer assigned
successively to idle set $A$ selects server $i$. For $i\in A$ and
$1\le\ell\le|A|$, the subset probabilities in
\eqref{eq:assignment-subset-probability} give the explicit formula
\begin{equation}\label{eq:wait-ordered-assignment}
 a_{i,\ell}(A)=
   \sum_{\substack{F\subseteq A\setminus\{i\}\\|F|=\ell-1}}
         K_{A,\ell-1}(F)\,p_i(A\setminus F).
\end{equation}
Outside these ranges set $a_{i,\ell}(A)=0$. The preceding customers
select $F$, leaving the tag to select from $A\setminus F$.
Thus $\sum_{i\in A}a_{i,\ell}(A)=1$; uniform assignment gives
$a_{i,\ell}(A)=1/|A|$.

Define the $j\times c$ matrix of transitions from waiting into service by
\begin{equation}\label{eq:wait-service-entry-matrix}
 (\mathsf C_j)_{\ell i}
    =\sum_{\substack{\chg{A\in\mathcal P:\,i\in A}\\|A|\ge\ell}}
                  \mu_A a_{i,\ell}(A),
 \qquad 1\le\ell\le j,\quad 1\le i\le c.
\end{equation}
Rows with $\ell>c$ are zero, and
$\mathsf C_j\mathbf1=\boldsymbol t_j$. A shock of size at least
$\ell$ admits the tagged customer as the $\ell$th assignment at
that event. With $\Delta=\operatorname{diag}(\nu_1,\ldots,\nu_c)$,
the conditional sojourn-time subgenerator is therefore
\begin{equation}\label{eq:sojourn-conditional-generator}
 \mathsf G_j=
 \begin{pmatrix}
   \mathsf U_j&\mathsf C_j\\
   0&-\Delta
 \end{pmatrix},
 \qquad
 \boldsymbol\alpha_j=(e_j^{\top},0_{1\times c}).
\end{equation}
The first $j$ phases are waiting positions, followed by $c$ service
phases with absorption rates $\nu_i$. Thus the conditional sojourn
law is $\mathrm{PH}(\boldsymbol\alpha_j,\mathsf G_j)$.
For immediate service from busy set $B\subsetneq\mathcal I$, its transform is
\begin{equation}\label{eq:sojourn-immediate-transform}
 \sum_{i\in\mathcal I\setminus B}p_i(\mathcal I\setminus B)
                          \frac{\nu_i}{s+\nu_i}.
\end{equation}

\subsection{A finite stationary sojourn-time representation}
\label{subsec:stationary-sojourn}

For each server $i$, define
\begin{align}
 d_i&=\sum_{B\subseteq\mathcal I\setminus\{i\}}
                \pi_B p_i(\mathcal I\setminus B),
 \label{eq:sojourn-immediate-mass}\\
 g_i&=\sum_{\chg{A\in\mathcal P:\,i\in A}}\mu_A
             \sum_{\ell=1}^{|A|}(1-r)r^{\ell-1}a_{i,\ell}(A).
 \label{eq:sojourn-delayed-rate}
\end{align}
Here $\pi_{\emptyset}=\pi_0$. The quantity $d_i$ is the probability
of immediate service on server $i$, and $g_i$ is the rate of entry
to server $i$ under the geometric distribution of the remaining
waiting position. All sums are finite, and
\begin{equation}\label{eq:sojourn-row-sums}
 \sum_i d_i=1-P_{\rm full},\qquad
 \sum_i g_i=\vartheta.
\end{equation}
Under uniform random assignment, the delayed-entry rates simplify to
\begin{equation}\label{eq:sojourn-uniform-rates}
 g_i=\sum_{\chg{A\in\mathcal P:\,i\in A}}\frac{\mu_A}{|A|}(1-r^{|A|}).
\end{equation}

\begin{theorem}[Stationary sojourn time]
\label{thm:stationary-sojourn}
The server-marked waiting-time law is
\begin{equation}\label{eq:sojourn-joint-entry}
 \mathbb P(W_q=0,I_*=i)=d_i,\qquad
 \mathbb P(W_q\in dt,I_*=i)
      =P_{\rm full}g_i e^{-\vartheta t}\,dt,\quad t>0.
\end{equation}
In particular, conditional on positive waiting, $W_q$ and $I_*$
are independent and $\mathbb P(I_*=i\mid W_q>0)=g_i/\vartheta$.
For $s,z\ge0$,
\begin{equation}\label{eq:sojourn-bivariate-transform}
 \mathbb E[e^{-sW_q-zS}\mathbf1_{\{I_*=i\}}]
    =\left(d_i+\frac{P_{\rm full}g_i}{s+\vartheta}\right)
                         \frac{\nu_i}{z+\nu_i}.
\end{equation}
Consequently, $T$ has a phase-type representation with $c+1$ phases:
\begin{equation}\label{eq:sojourn-stationary-PH}
 T\sim\mathrm{PH}(\boldsymbol\alpha,\mathsf G),\qquad
 \boldsymbol\alpha=(P_{\rm full},d_1,\ldots,d_c),\qquad
 \mathsf G=
 \begin{pmatrix}
    -\vartheta&\boldsymbol g\\
    0_{c\times1}&-\Delta
 \end{pmatrix},
\end{equation}
where $\boldsymbol g=(g_1,\ldots,g_c)$. Its transform is
\begin{equation}\label{eq:sojourn-stationary-transform}
 \mathbb E[e^{-sT}]=
    \sum_{i=1}^c\left(d_i+\frac{P_{\rm full}g_i}{s+\vartheta}\right)
                                 \frac{\nu_i}{s+\nu_i}.
\end{equation}
\end{theorem}
\begin{proof}
Let $\mathsf C$ be the infinite-by-$c$ matrix with entries in
\eqref{eq:wait-service-entry-matrix}; only its first $c$ rows can
be nonzero. The geometric row $\boldsymbol a$ in
\eqref{eq:wait-quasistationary} satisfies
$\boldsymbol a\mathsf C=\boldsymbol g$. Hence the row of
service-entry densities, conditional on positive waiting, is
\[
 \boldsymbol a e^{\mathsf Ut}\mathsf C
       =e^{-\vartheta t}\boldsymbol g.
\]
Together with the immediate-service masses, this proves
\eqref{eq:sojourn-joint-entry}. The strong Markov property at
service initiation gives \eqref{eq:sojourn-bivariate-transform}.
One transient waiting phase with rates $g_i$ into the service
phases gives \eqref{eq:sojourn-stationary-PH}. Finally, put $z=s$
in \eqref{eq:sojourn-bivariate-transform} and sum over $i$.
\end{proof}

The $c+1$-phase bound concerns one customer's stationary sojourn
distribution. It does not bound the conditional representation
\eqref{eq:sojourn-conditional-generator} for a fixed initial queue
state or the heterogeneous boundary computation, and it does not
imply independence of successive delays. The fixed rule $p_i(E)$
ensures that the tagged customer's server choice is unaffected by
the number waiting behind it. Queue-population-dependent policies,
such as those in Remark~\ref{rem:assignment-generalization}, require
a separate analysis before this stationary reduction can be asserted.

For direct evaluation of tail probabilities, put
\begin{equation}\label{eq:sojourn-two-exponential-tail}
 H_{a,b}(t)=
 \begin{cases}
 \displaystyle\frac{b e^{-at}-a e^{-bt}}{b-a},&a\ne b,\\[1ex]
 (1+at)e^{-at},&a=b.
 \end{cases}
\end{equation}
This is the survival function of the sum of independent exponential
variables of rates $a$ and $b$. The exact stationary sojourn tail is
\begin{equation}\label{eq:sojourn-stationary-tail}
 \mathbb P(T>t)=\boldsymbol\alpha e^{\mathsf Gt}\mathbf1
   =\sum_{i=1}^c d_i e^{-\nu_it}
       +P_{\rm full}\sum_{i=1}^c\frac{g_i}{\vartheta}
                                H_{\vartheta,\nu_i}(t).
\end{equation}
This formula also covers coincident rates. All moments follow from
\begin{equation}\label{eq:sojourn-moments}
 \mathbb E[T^m]=m!\,\boldsymbol\alpha(-\mathsf G)^{-m}\mathbf1,
 \qquad m\ge1.
\end{equation}

Write $\varpi_i=d_i+P_{\rm full}g_i/\vartheta$ for the probability
that a stationary arrival is served by server $i$. In particular,
\begin{equation}\label{eq:sojourn-mean}
 \mathbb E T=\frac{P_{\rm full}}{\vartheta}
                    +\sum_{i=1}^c\frac{\varpi_i}{\nu_i}.
\end{equation}
The busy probability and stationary server-flow identity are
\begin{equation}\label{eq:sojourn-server-flow}
 u_i=P_{\rm full}+\sum_{\substack{B\subsetneq\mathcal I\\i\in B}}\pi_B,
 \qquad \lambda\varpi_i=\nu_i u_i.
\end{equation}
Every service initiation has one completion, whose rate while busy is
$\nu_i$. Consequently,
\begin{equation}\label{eq:sojourn-Little}
 \lambda\mathbb E S=\sum_i u_i=\mathbb E_\pi\min(X,c),
 \qquad \lambda\mathbb E T=\mathbb E_\pi X.
\end{equation}
These identities and \eqref{eq:wait-Little} check the boundary and delay
computations independently.

\begin{corollary}[Equal marginal service rates and the independent case]
\label{cor:sojourn-equal-rates}
If $\nu_i=\nu$ for all $i$, then $S$ is exponential of rate $\nu$
and independent of $W_q$, and
\begin{equation}\label{eq:sojourn-equal-rates}
 \mathbb E[e^{-sT}]
   =\left(1-P_{\rm full}
               +P_{\rm full}\frac{\vartheta}{s+\vartheta}\right)
                           \frac{\nu}{s+\nu}.
\end{equation}
The stationary sojourn law then needs only two phases. This applies
in particular to homogeneous Marshall--Olkin servers, for which
$\nu=\sum_{k=1}^c\binom{c-1}{k-1}\mu_{(k)}$.
For independent heterogeneous servers,
$\vartheta=\Theta-\lambda$ and
$\mathbb P(I_*=i\mid W_q>0)=\nu_i/\Theta$.
For the classical homogeneous $M/M/c$ queue, the waiting-time
formula reduces to
\begin{equation}\label{eq:wait-classical}
 \mathbb P(W_q>t)=P_{\rm full}e^{-(c\nu-\lambda)t},
\end{equation}
where $P_{\rm full}$ is the usual Erlang-C delay probability.
\end{corollary}
\begin{proof}
When the rates are equal, summing
\eqref{eq:sojourn-bivariate-transform} over $i$ factors the joint
transform into the waiting transform and $\nu/(z+\nu)$.
For independent servers, $R_1=\Theta$, $R_k=0$ for $k\ge2$,
and $r=\lambda/\Theta$. Equations~\eqref{eq:wait-rate} and
\eqref{eq:sojourn-delayed-rate} then give
$\vartheta=\Theta(1-r)$ and $g_i=(1-r)\nu_i$.
\end{proof}

\subsection{Two-server formulas and performance}
\label{subsec:waiting-two-server}

For $c=2$, use $u=\mu_{12}$ and $p=p_1(\{1,2\})$ as in
Section~\ref{subsec:departure-two}. Then
\begin{equation}\label{eq:wait-two-rate}
 \vartheta=(\mu_1+\mu_2)(1-r)+u(1-r^2),\qquad
 \nu_1=\mu_1+u,\quad\nu_2=\mu_2+u.
\end{equation}
The first two service-entry rows are
\begin{equation}\label{eq:sojourn-two-entry}
 \mathsf C_2=
 \begin{pmatrix}
   \mu_1+pu&\mu_2+(1-p)u\\
   (1-p)u&pu
 \end{pmatrix}.
\end{equation}
After a two-server shock, a tag in position $1$ chooses server $1$
with probability $p$, whereas a tag in position $2$ receives it with
probability $1-p$. The stationary parameters are
\begin{align}
 d_1&=p\pi_0+\pi_{\{2\}},&
 d_2&=(1-p)\pi_0+\pi_{\{1\}},
 \label{eq:sojourn-two-immediate}\\
 g_1&=(1-r)\bigl[\mu_1+u\{p+r(1-p)\}\bigr],&
 g_2&=(1-r)\bigl[\mu_2+u\{1-p+rp\}\bigr].
 \label{eq:sojourn-two-delayed}
\end{align}
Substitution in \eqref{eq:sojourn-stationary-PH} gives the three phases:
waiting, service at server $1$, and service at server $2$.

\begin{example}[Output and delays at fixed marginal service rates]
\label{exa:waiting-two-performance}\label{exa:dep-two-performance}
\normalfont
Fix $\lambda=1.5$, $p=1/2$, $\nu_1=1$, and $\nu_2=1.5$, and vary
$u$ through $\mu_1=1-u$, $\mu_2=1.5-u$. Capacity remains $2.5$.
The departure and delay formulas give
\begin{center}
\begin{tabular}{c|ccccc}
 $u$&$P_{\rm full}$&$\mathbb E W_q$&$\mathbb E T$
       &$\operatorname{SCV}(\tau_d)$&$I_D(0+)$\\\hline
 $0$    &$0.45455$&$0.45455$&$1.26263$&$1.0000$&$1.0000$\\
 $0.25$ &$0.46525$&$0.51360$&$1.32125$&$0.9536$&$1.1551$\\
 $0.50$ &$0.47511$&$0.57738$&$1.38462$&$0.9030$&$1.3167$
\end{tabular}
\end{center}
Across these cases, throughput stays at $1.5$, while delay probability
and mean delays increase with $u$. Departure-gap SCV decreases and
short-interval count dispersion increases; these measure different
aspects of output variability. In every case, $I_D(t)\to1$.
These comparisons do not assert general monotonicity in dependence.
\end{example}

\section{Dependence comparisons and numerical illustrations}
\label{sec:numerical-study}

We compare two three-server families with identical bivariate
residual-service laws, then examine heterogeneity, assignment, and
output variability at fixed capacity. The finite boundary and delay
formulas of Sections~\ref{sec:heterogeneous-MAM}--\ref{sec:waiting-sojourn}
use the full geometric tail. Only finite-time departure variability
uses the queue truncation checked below.

\subsection{Dependence with matched bivariate residual-service laws}
\label{subsec:numerical-matched}

Fix marginal service rates $(\nu_1,\nu_2,\nu_3)$ and choose
$0\le q\le\frac12\min_i\nu_i$. \chg{The parameter $q$ used in this section
should not be confused with the subscript in $W_q$, which denotes queueing time.}
Consider two families:
\[
\begin{array}{c|ccc}
 &\mu_{\{i\}}&\mu_{\{i,j\}}\ (i<j)&\mu_{\{1,2,3\}}\\\hline
 \text{Pairwise shocks}&\nu_i-2q&q&0\\
 \text{Three-server shocks}&\nu_i-q&0&q
\end{array}
\]
\chg{Both families have the same marginal rates $\nu_i$} and total shock rate $q$
affecting each specified pair together. Their residual times at a
common observation time therefore have identical bivariate survival laws
\[
 \mathbb P(X_i>s,X_j>t)
   =\exp\{-\nu_i s-\nu_j t+q\min(s,t)\},\qquad s,t\ge0,
\]
and correlation $q/(\nu_i+\nu_j-q)$. These are residual variables
from a common time, not full services begun at different times.
Every bivariate law is fixed while the three-server dependence changes.

Write $\Theta=\nu_1+\nu_2+\nu_3$. The aggregate fully busy
completion rates are
\[
 (R_1,R_2,R_3)=
 \begin{cases}
 (\Theta-6q,3q,0),&\text{pairwise shocks},\\
 (\Theta-3q,0,q),&\text{three-server shocks}.
 \end{cases}
\]
Both the mean rate $\sum_k kR_k=\Theta$ and the variance rate
$\sum_k k^2R_k=\Theta+6q$ of the saturated potential completion
count agree. Nevertheless, their stationary conditional waiting laws
have a strict ordering.

\begin{proposition}[Conditional waiting-time ordering]
\label{prop:matched-waiting-order}
Suppose $\nu_i>0$, $0<q\le\frac12\min_i\nu_i$, and
$0<\lambda<\Theta$. For the two families above, use superscripts
$\mathrm{pair}$ and $\mathrm{three}$ to distinguish their stationary
quantities. Then
\begin{equation}\label{eq:matched-rate-order}
 r^{\mathrm{pair}}>r^{\mathrm{three}},\qquad
 \vartheta^{\mathrm{pair}}<\vartheta^{\mathrm{three}}.
\end{equation}
Consequently, for every $t>0$,
\[
 \mathbb P(W_q^{\mathrm{pair}}>t\mid W_q^{\mathrm{pair}}>0)
 >\mathbb P(W_q^{\mathrm{three}}>t\mid W_q^{\mathrm{three}}>0).
\]
These comparisons hold for homogeneous or heterogeneous marginal
rates and for every fixed idle-set assignment rule $p_i(E)$; the
rules may differ between the two systems.
\end{proposition}
\begin{proof}
The root equation \eqref{eq:hetero-MAM-root} is
$\lambda=F_{\mathrm{pair}}(r)$ or
$\lambda=F_{\mathrm{three}}(r)$, where
\begin{equation}\label{eq:matched-root-functions}
 \begin{aligned}
 F_{\mathrm{pair}}(x)&=(\Theta-3q)x+3qx^2,\\
 F_{\mathrm{three}}(x)&=(\Theta-2q)x+qx^2+qx^3.
 \end{aligned}
\end{equation}
Both functions increase strictly from $0$ to $\Theta$ on $[0,1]$.
Their coefficients are nonnegative and their linear terms are positive.
Moreover,
\[
 F_{\mathrm{three}}(x)-F_{\mathrm{pair}}(x)
     =qx(1-x)^2>0,\qquad 0<x<1.
\]
Their unique roots at height $\lambda$ therefore satisfy
$r^{\mathrm{pair}}>r^{\mathrm{three}}$.
Since $\lambda(1-r)/r$ is strictly decreasing in $r$, equation
\eqref{eq:wait-rate} gives the rate ordering. Theorem~
\ref{thm:stationary-waiting} then gives the conditional survival
ordering. Neither root equation involves the assignment rule.
\end{proof}

At $q=0$, the shock mechanisms coincide, and both comparisons in
\eqref{eq:matched-rate-order} become equalities. For $q>0$, write
$d_{\mathrm{pair}}=P_{\rm full}^{\mathrm{pair}}$ and
$d_{\mathrm{three}}=P_{\rm full}^{\mathrm{three}}$.
Equation~\eqref{eq:wait-stationary-law} gives the exact tail ratio
\begin{equation}\label{eq:matched-tail-ratio}
 \frac{\mathbb P(W_q^{\mathrm{pair}}>t)}
      {\mathbb P(W_q^{\mathrm{three}}>t)}
 =\frac{d_{\mathrm{pair}}}{d_{\mathrm{three}}}
   e^{(\vartheta^{\mathrm{three}}-\vartheta^{\mathrm{pair}})t},
 \qquad t\ge0.
\end{equation}
If $d_{\mathrm{pair}}\ge d_{\mathrm{three}}$, the pairwise family
has the larger waiting tail for every $t>0$. If
$d_{\mathrm{pair}}<d_{\mathrm{three}}$, the tails cross exactly once,
at
\begin{equation}\label{eq:matched-tail-crossing}
 t_*=
 \frac{\log(d_{\mathrm{three}}/d_{\mathrm{pair}})}
      {\vartheta^{\mathrm{three}}-\vartheta^{\mathrm{pair}}}>0:
\end{equation}
the pairwise tail is smaller for $0\le t<t_*$ and larger for
$t>t_*$. Its unconditional mean waiting time is larger precisely
when
\[
 \frac{d_{\mathrm{pair}}}{d_{\mathrm{three}}}
   >\frac{\vartheta^{\mathrm{pair}}}{\vartheta^{\mathrm{three}}}.
\]
Thus the root ordering fixes the conditional waiting law, while
comparisons of delay probabilities and unconditional means also
require the boundary solution.

For the numerical comparisons, take $\nu_i=1$ and uniform assignment.
Then $\Theta=3$. We use arrival rates $\lambda=1.5$ and $2.4$,
giving loads $\lambda/\Theta=0.5$ and $0.8$, and evaluate
$q=0,0.025,\ldots,0.45$.

\begin{figure}[htbp]
\centering
\includegraphics[width=\textwidth]{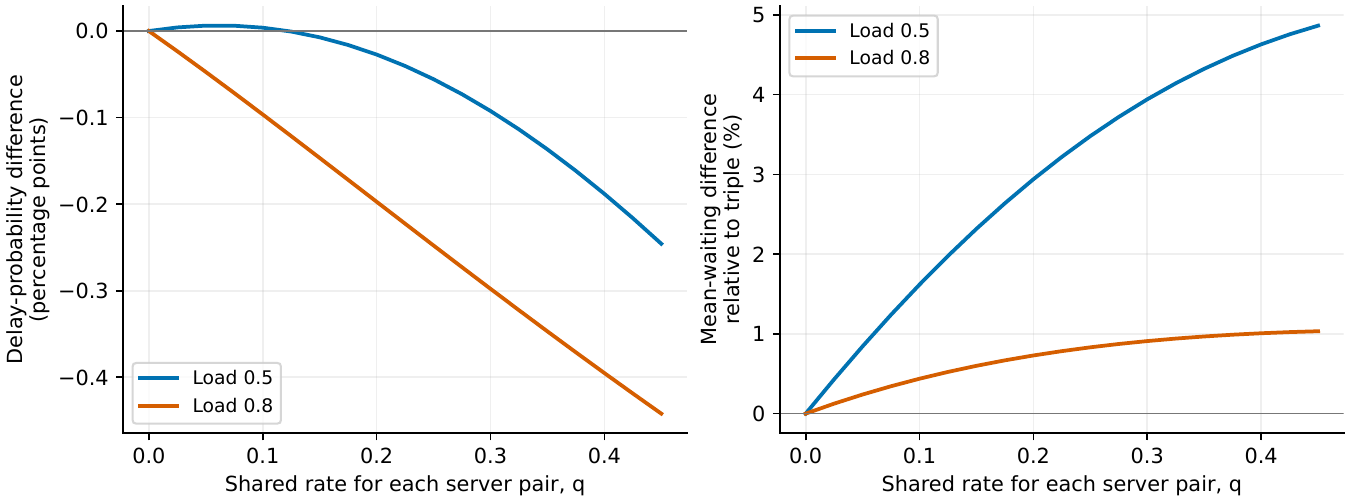}
\caption{Matched bivariate residual-service laws for homogeneous
servers. The left panel shows the pairwise-family delay probability
minus the three-server-family probability, in percentage points. The
right panel shows $100(\mathbb E W_q^{\rm pair}/
\mathbb E W_q^{\rm three}-1)$. The horizontal variable $q$ is the
rate of shocks affecting each specified pair together.}
\label{fig:numerical-dependence}
\end{figure}

At $\lambda=2.4$ and $q=0.25$, each pair has correlation $1/7$.
Table~\ref{tab:numerical-delays} gives a lower delay probability but
higher mean wait for pairwise shocks: the probability ratio $0.99630$
exceeds the waiting-rate ratio $0.98809$, and the tails cross at
$t_*\approx0.64523$. Figure~\ref{fig:numerical-dependence} also shows
the delay-probability difference changing sign with $q$ at load $0.5$.
Thus the strict conditional ordering does not determine the ordering
of delay probabilities or unconditional means without the boundary masses.

\begin{table}[htbp]
\centering\small
\caption{Three-server delay performance at $\lambda=2.4$.
The marginal rates are $(1,1,1)$ or $(0.6,1,1.4)$; all dependent
rows use $q=0.25$, while the independent row uses $q=0$.}
\label{tab:numerical-delays}
\begin{tabular}{lllrrrr}
\hline
Rates & Shocks & Assignment &$P_{\rm full}$&$\vartheta$
 &$\mathbb E W_q$&$\mathbb E T$\\\hline
Homogeneous & Independent & Uniform & 0.64719 & 0.60000 & 1.07865 & 2.07865 \\
Homogeneous & Pairwise & Uniform & 0.66799 & 0.47589 & 1.40365 & 2.40365 \\
Homogeneous & Three-server & Uniform & 0.67047 & 0.48163 & 1.39208 & 2.39208 \\
Heterogeneous & Pairwise & Uniform & 0.67739 & 0.47589 & 1.42340 & 2.43478 \\
Heterogeneous & Pairwise & Fastest & 0.65639 & 0.47589 & 1.37929 & 2.36489 \\
Heterogeneous & Three-server & Uniform & 0.67986 & 0.48163 & 1.41157 & 2.42287 \\
Heterogeneous & Three-server & Fastest & 0.65905 & 0.48163 & 1.36838 & 2.35363 \\
\hline
\end{tabular}
\end{table}

\subsection{Heterogeneity and assignment}
\label{subsec:numerical-assignment}

Use $(\nu_1,\nu_2,\nu_3)=(0.6,1,1.4)$, $\lambda=2.4$, and $q=0.25$.
Within each shock family, compare uniform assignment with sequential
assignment to the idle server having the largest $\nu_i$.
Both are fixed idle-set rules. Each family retains the same aggregate
rates $R_k$, ratio $r$, and positive-wait rate $\vartheta$ across the
rate profiles and assignment rules; only boundary and service-entry
weights change.

With three-server shocks, fastest-idle assignment lowers mean waiting
from $1.41157$ to $1.36838$ and mean sojourn from $2.42287$ to $2.35363$.
The relative waiting-tail reduction is constant in $t$, since only
$P_{\rm full}$ changes; the sojourn-tail reduction also reflects service
selection. These examples do not establish fastest-idle optimality.

\begin{figure}[htbp]
\centering
\includegraphics[width=\textwidth]{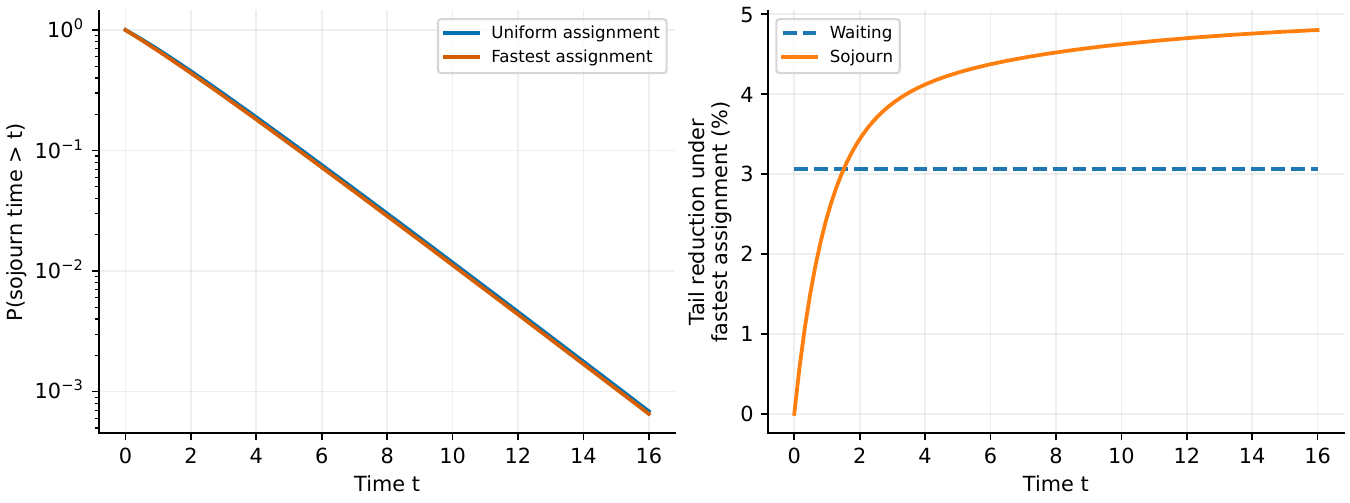}
\caption{Heterogeneous servers with three-server shocks,
$q=0.25$, and $\lambda=2.4$. Left: stationary sojourn tails.
Right: percentage reduction in the waiting and sojourn tails under
fastest-idle assignment relative to uniform assignment.}
\label{fig:numerical-assignment}
\end{figure}

\subsection{Departure variability across time scales}
\label{subsec:numerical-output}

At $\lambda=2.4$ with homogeneous servers, compare independent services
with the matched families at $q=0.4$. Table~\ref{tab:numerical-output}
shows unchanged throughput, but different epoch rates, batch frequencies,
and gap SCVs. Both dependent cases have gap SCV below one and
short-interval count dispersion above one.

\begin{table}[htbp]
\centering
\caption{Departure performance for homogeneous three-server systems.
The dependent rows use $q=0.4$; all rows use $\lambda=2.4$.}
\label{tab:numerical-output}
\begin{tabular}{lrrrr}
\hline
Shocks &$\beta$&$\mathbb P^{\rm dep}(K\ge2)$
 &$\operatorname{SCV}(\tau_d)$&$I_D(0+)$\\\hline
Independent & 2.40000 & 0.00000 & 1.00000 & 1.00000 \\
Pairwise & 1.53944 & 0.55901 & 0.90968 & 1.71713 \\
Three-server & 1.80625 & 0.17781 & 0.94745 & 1.72195 \\
\hline
\end{tabular}
\end{table}

\begin{figure}[htbp]
\centering
\includegraphics[width=0.82\textwidth]{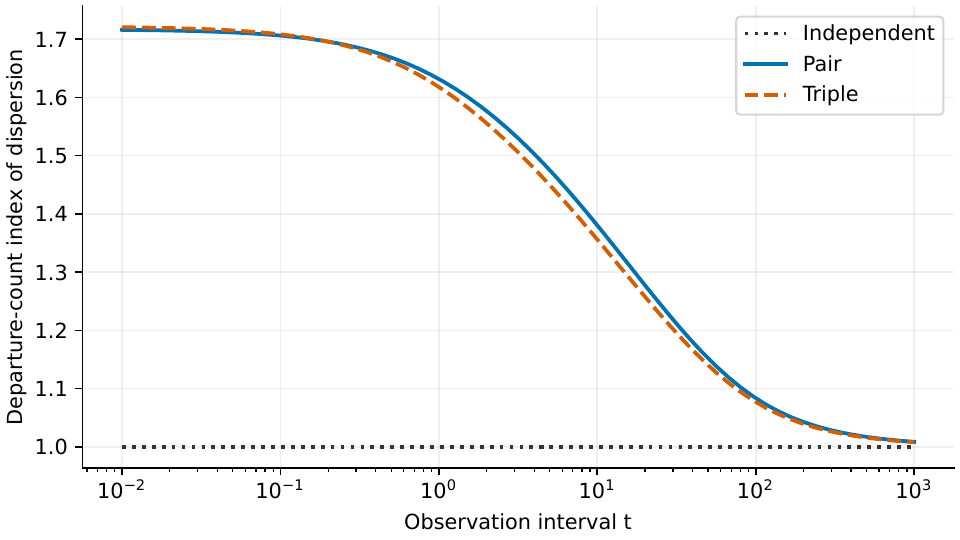}
\caption{Departure-count dispersion across observation intervals,
with the parameters of Table~\ref{tab:numerical-output}. The computed
curves approach the conservation limit $I_D(t)\to1$; equality at all
time scales holds for the independent homogeneous benchmark.}
\label{fig:numerical-dispersion}
\end{figure}

\subsection{Computational checks and simulation}
\label{subsec:numerical-validation}

The implementation reproduces Erlang-C waiting and homogeneous
independent Poisson output, and matches the two-server example to its
stated rounding. Generator, normalization, server-flow, and Little-law
residuals are below $10^{-10}$.

For Figure~\ref{fig:numerical-dispersion}, suppressing arrivals at
$N=136,186,184$ for the independent, pairwise, and three-server cases
omits stationary mass below $10^{-13}$. Raising each cutoff by $40$
changes $I_D(t)$ by less than $1.3\times10^{-11}$ at six times between
$0.01$ and $1000$. This is a sensitivity check, not a certified
variance-error bound.

A separate event-driven simulator follows arrivals, labelled shocks,
and assignments. Four cases at $q=0.25$ each use 16 independent
replications, a 5000-unit warm-up, and a 50000-unit arrival window.
Every observed arrival is followed to completion, including departures
beyond the window. Of 32 individual 95\% Student-$t$ intervals,
31 cover the analytical value. The exception is homogeneous pairwise
$I_D(1)$: value $1.38699$, interval $[1.37826,1.38676]$.
Detailed estimates, seeds, parameters, and code accompany the paper.

\section{Conclusion}

The explicit QBD rate matrix gives a common stationary foundation for
the $M/M_D/c$ queue state, departures, and customer delays. One scalar
root and a finite boundary system satisfy boundary and tail equations
together; homogeneous boundaries also admit a scalar recursion.
Completion-rate consistency specifies when the Markovian service
mechanism has a Marshall--Olkin shock representation.

The geometric tail yields exponential waiting conditional on delay
and a stationary sojourn law with at most $c+1$ phases. Boundary
weights remain necessary even for these small representations.
Shared shocks produce batch departures and short-interval overdispersion,
while conservation fixes the long-run variance rate without implying
Poisson output. Matched bivariate residual laws can still give different
waiting performance: conditional decay and boundary delay probability
jointly determine the unconditional comparison.

Further work concerns temporal dependence of virtual waiting and
successive delays, extensions with specified shock and routing rules,
and many-server or mean-field limits.

\appendix
\section{Proof of the Marshall--Olkin characterization}
\label{app:MO-characterization}

We prove Theorem~\ref{thm:MO-characterization}.

\begin{proof}
Sufficiency follows from Lemma~\ref{rem:part-of-lack-of-memory-2}.
For necessity, write $\mathcal I=\{1,\ldots,c\}$ and put
$g_B(x_B)=-\log\bar F_B(x_B)$. The assumed property gives
\[
 g_B(x_B+t\mathbf1_B)=g_B(x_B)+g_B(t\mathbf1_B).
\]
Taking $x_B=u\mathbf1_B$ shows that
$h_B(t)=g_B(t\mathbf1_B)$ is additive on $[0,\infty)$:
$h_B(u+t)=h_B(u)+h_B(t)$. Continuity therefore implies
\[
 g_B(t\mathbf1_B)=\Lambda(B)t,\qquad
 g_B(x_B+t\mathbf1_B)=g_B(x_B)+\Lambda(B)t
\]
for a finite $\Lambda(B)\ge0$. Set $\Lambda(\emptyset)=0$.

Write $g=g_{\mathcal I}$. Order the coordinates of $x$ as
$0\le x_{i_1}\le\cdots\le x_{i_c}$, set $x_{i_0}=0$, and let
$B_k=\{i_k,\ldots,i_c\}$. Successively subtracting the smallest
remaining coordinate and using $\mathbb P(X_i>0)=1$ to remove zero
coordinates gives
\begin{equation}\label{eq:MO-ordered-exponent}
 g(x)=\sum_{k=1}^c\Lambda(B_k)(x_{i_k}-x_{i_{k-1}})
     =\int_0^\infty\Lambda(\{i:x_i>u\})\,du.
\end{equation}

It remains to recover nonnegative shock rates. For each nonempty
$A\subseteq\mathcal I$, define the probability that exactly the
components in $A$ have failed by time $h$:
\[
 p_A(h)=\mathbb P(X_i\le h\ (i\in A),\ X_j>h\ (j\in A^c)),
 \qquad A^c=\mathcal I\setminus A.
\]
Inclusion--exclusion and the diagonal survival probabilities imply
\[
 p_A(h)=\sum_{D\subseteq A}(-1)^{|D|}
                     e^{-h\Lambda(A^c\cup D)}.
\]
This is a finite sum of differentiable functions, with $p_A(0)=0$.
Consequently,
\begin{equation}\label{eq:MO-rate-recovery}
 \mu_A:=\lim_{h\downarrow0}\frac{p_A(h)}{h}
       =-\sum_{D\subseteq A}(-1)^{|D|}\Lambda(A^c\cup D)
       \ge0.
\end{equation}
The inequality follows directly from $p_A(h)\ge0$.
For every nonempty $B$, partitioning the event that at least one
component in $B$ has failed gives
\[
 1-e^{-h\Lambda(B)}
     =\sum_{\chg{A\in\mathcal P:\,A\cap B\ne\emptyset}}p_A(h).
\]
Divide by $h$ and let $h\downarrow0$ to obtain
\begin{equation}\label{eq:MO-rate-restriction}
 \Lambda(B)=\sum_{\chg{A\in\mathcal P:\,A\cap B\ne\emptyset}}\mu_A.
\end{equation}
Substitution into \eqref{eq:MO-ordered-exponent} and interchange of
the finite sum and integral now yield
\[
 g(x)=\sum_{\emptyset\ne A\subseteq\mathcal I}\mu_A
       \int_0^\infty
          \mathbf1_{\{A\cap\{i:x_i>u\}\ne\emptyset\}}\,du
     =\sum_{\emptyset\ne A\subseteq\mathcal I}
          \mu_A\max_{i\in A}x_i.
\]
This proves the asserted survival function. Formula
\eqref{eq:MO-rate-recovery} determines the rates uniquely. Finally,
$\mathbb P(X_i>t)=e^{-t\Lambda(\{i\})}$ tends to zero because
$X_i$ is finite almost surely, so
$\sum_{\chg{A\in\mathcal P:\,i\in A}}\mu_A=\Lambda(\{i\})>0$.
\end{proof}

\section*{Acknowledgments}
The authors used OpenAI's ChatGPT to assist with language editing, organization, exposition, and presentation of the manuscript. The authors reviewed all generated suggestions and take full responsibility for the content of the paper.

\begin{chgblock}
\section*{Funding Information}
Q. Bu acknowledges support, in part, from the Natural Science Foundation of
Jiangsu Province of China (BK20240605) and the Natural Science Foundation of
Nanjing University of Posts and Telecommunications (Grant No.~NY224126).
S. Thapa acknowledges support from Carleton University. Y. Q. Zhao acknowledges
support, in part, from the Natural Sciences and Engineering Research Council of
Canada (Grant RGPIN-2022-04915).
\end{chgblock}


\begin{thebibliography}{99}




%





\bibitem{BernhartEtAl2015}
Bernhart, G., Fern{\'a}ndez, L., Mai, J.-F., Schenk, S. and Scherer, M. (2015)
A survey of dynamic representations and generalizations of the Marshall--Olkin distribution,
in Cherubini, U., Durante, F. and Mulinacci, S. (Eds.),
\emph{Marshall--Olkin Distributions: Advances in Theory and Applications},
Springer Proceedings in Mathematics \& Statistics, Vol.~141, 1--13.






%









\begin{chgblock}
\bibitem{He:2014}
He, Q.-M. (2014)
\textit{Fundamentals of Matrix-Analytic Methods},
Springer, New York.

\bibitem{Kelly:1982QR}
Kelly, F.P. (1982)
Networks of quasi-reversible nodes,
in R.L. Disney and T.J. Ott (eds.),
\textit{Applied Probability--Computer Science: The Interface}, Vol.~1,
Birkh\"auser, Boston, 3--29.
\end{chgblock}

\bibitem{Marshall-Olkin:1967a}
Marshall, A.W. and Olkin, I. (1967)
A multivariate exponential distribution,
\textit{Journal of the American Statistical Association}.
\textbf{62}, 30--44.



\bibitem{Mitchell et al:1977}
Mitchell, C.R., Paulson, A.S. and Beswick, C.A. (1977)
The effect of correlated exponential service times on single server tandem queues,
\textit{Naval Research Logistics Quarterly},
\textbf{24(1)}, 95--112.

\bibitem{Neuts:1981}
Neuts, M.F. (1981)
\textit{Matrix-Geometric Solutions in Stochastic Models: An Algorithmic Approach},
Johns Hopkins University Press.

%








\bibitem{Singh:1970}
Singh, V.P. (1970)
Two-server Markovian queues with balking: heterogeneous vs. homogeneous servers,
\textit{Operations Research},
\textbf{18}, 145--159.

\bibitem{Thapa-Zhao:2026}
Thapa, S. and Zhao, Y.Q. (2026)
Markov modeling approach for queues with correlated service times: the $M/M_D/2$ model,
\textit{Queueing Systems}, \textbf{110}, Article 38.
\url{https://doi.org/10.1007/s11134-026-09998-x}.




\bibitem{Bu-Thapa-Zhao:2026}
Bu, Q., Thapa, S. and Zhao, Y.Q. (2026)
Queues with correlated service times --- the $M/M_D/c$ model,
\textit{arXiv preprint}, arXiv:2606.24881v1.
Preliminary version of the present work.
\url{https://arxiv.org/abs/2606.24881v1}.

\bibitem{Brigo-Mai-Scherer:2016}
Brigo, D., Mai, J.-F. and Scherer, M. (2016)
Markov multi-variate survival indicators for default simulation as a new
characterization of the Marshall--Olkin law,
\textit{Statistics \& Probability Letters}, \textbf{114}, 60--66.
\url{https://doi.org/10.1016/j.spl.2016.03.013}.

\bibitem{Burke:1956}
Burke, P.J. (1956)
\href{https://doi.org/10.1287/opre.4.6.699}{The output of a queuing system},
\textit{Operations Research}, \textbf{4(6)}, 699--704.

\bibitem{Lucantoni:1991}
Lucantoni, D.M. (1991)
\href{https://doi.org/10.1080/15326349108807174}{New results on the single server queue with a batch Markovian arrival process},
\textit{Communications in Statistics---Stochastic Models},
\textbf{7(1)}, 1--46.

\bibitem{Wolff:1982}
Wolff, R.W. (1982)
\href{https://doi.org/10.1287/opre.30.2.223}{Poisson arrivals see time averages},
\textit{Operations Research}, \textbf{30(2)}, 223--231.

\bibitem{Little:1961}
Little, J.D.C. (1961)
\href{https://doi.org/10.1287/opre.9.3.383}{A proof for the queuing formula: $L=\lambda W$},
\textit{Operations Research}, \textbf{9(3)}, 383--387.


\bibitem{Ramaswami-Latouche:1986}
Ramaswami, V. and Latouche, G. (1986)
A general class of Markov processes with explicit matrix-geometric solutions,
\textit{OR Spektrum}, \textbf{8}, 209--218.
\url{https://doi.org/10.1007/BF01721131}.

\bibitem{vanLeeuwaarden-Winands:2006}
van Leeuwaarden, J.S.H. and Winands, E.M.M. (2006)
Quasi-birth-and-death processes with an explicit rate matrix,
\textit{Stochastic Models}, \textbf{22(1)}, 77--98.
\url{https://doi.org/10.1080/15326340500481747}.

\bibitem{Pang-Whitt:2012}
Pang, G. and Whitt, W. (2012)
Infinite-server queues with batch arrivals and dependent service times,
\textit{Probability in the Engineering and Informational Sciences},
\textbf{26(2)}, 197--220.
\url{https://doi.org/10.1017/S0269964811000337}.

\bibitem{Medhi:1975}
Medhi, J. (1975)
Waiting time distribution in a Poisson queue with a general bulk service rule,
\textit{Management Science}, \textbf{21(7)}, 777--782.
\url{https://doi.org/10.1287/mnsc.21.7.777}.

\end{thebibliography}
\end{document}